\documentclass[12pt]{amsart}

\usepackage{caption,amsmath,amssymb,amsfonts,amsthm,latexsym,graphicx,multirow,color,hyperref,enumerate}
\usepackage[all]{xy}
\usepackage{mathrsfs}
\usepackage{amsmath}
\usepackage{latexsym}
\usepackage{amssymb}
\usepackage{amsfonts}
\usepackage{longtable}
\usepackage{colortbl}
\include{PDF}
\usepackage{graphics}
\usepackage{ulem}
\usepackage{soul}
\usepackage{booktabs}

\def\ov{\overline}
\def\l{\langle} 
\def\r{\rangle} 

\def\ZZ{{\sf Z}}

\def\mod{{\sf mod~}}
\def\Aut{{\sf Aut}}

\def\Hom{{\sf Hom}}
\def\Hol{{\sf Hol}}

\def\D{{\rm D}} \def\Q{{\rm Q}}
\def\S{{\rm S}} 
 \def\M{{\rm M}}
\def\SD{{\rm SD}}
\def\Z{{\bf Z}}

\def\a{\alpha} \def\b{\beta} 
\def\s{\sigma}

\def\GL{{\rm GL}}

\def\calM{{\mathcal M}}
\def\vs{\vskip0.03in}
\def\le{\leqslant}
\def\ge{\geqslant}
\def\calT{{\mathcal T}}
\def\tt{{\sf t}}
\def\EM{{\sf EM}}

\newtheorem{theorem}{Theorem}[section]%
\newtheorem{lemma}[theorem]{Lemma}%
\newtheorem{corollary}[theorem]{Corollary}%
\newtheorem{proposition}[theorem]{Proposition}%
\newtheorem{definition}[theorem]{Definition}%
\newtheorem{hypothesis}[theorem]{Hypothesis}%

\def\qed{{\hfill$\Box$\smallskip}
\medbreak}

\begin{document}

\title[Groups of square-free characteristic]{Finite 2-groups having a cyclic or dihedral maximal subgroup and arc-transitive maps}
\thanks{The project was partially supported by the NNSF of China (11931005).}
\author{Peice Hua}
\address{Shenzhen International Center for Mathematics\\
Southern  University of Science and Technology,  Shenzhen 518055,  P.R.China}
\email{huapc@pku.edu.cn}

\begin{abstract}
We classify all finite $ 2 $-groups that have a cyclic or dihedral maximal subgroup and determine their automorphism groups. Based on this result, we classify all pairs $ (G,\calM) $, such that $ G $ is a finite $ 2 $-group and $ \calM $ is a $ G $-arc-transitive map with Euler characteristic not being divisible by $ 4 $.

\end{abstract}

\maketitle

\date\today

\section{Introduction}

This is the second paper in a series aimed at extending several well-known characterizations of finite groups and establishing connections between the relevant groups and classification problems concerning specific maps.\vs

Cyclic and dihedral groups can be regarded in every sense as the simplest and most fundamental components among finite groups. In this series of papers, we investigate the groups satisfying Hypothesis \ref{hypo-0}. It generalizes earlier results concerning \textit{almost-Sylow cyclic groups}, namely, of which each Sylow subgroup is cyclic or dihedral (see \cite{Sylow-metac}). Note that, a maximal subgroup of a $ p $-group is always of index $ p $.
\begin{hypothesis}\label{hypo-0}
	{\rm Let $G$ be a finite group of which each Sylow subgroup has a cyclic or dihedral maximal subgroup.}
\end{hypothesis}

The study of such groups is strongly motivated by classification problems of edge-transitive maps with square-free Euler characteristic. A \textit{map} $ \calM=(V,E,F) $ is a $ 2 $-cell embedding of a graph into a closed surface with vertex set $ V $, edge set $ E $ and face set $ F $. An \textit{arc} is a directed edge. If $\Aut\calM$, the group of all automorphisms of $ \calM $, acts transitively on the set of edges or arcs of $ \calM $, then $\calM$ is respectively called {\it edge-transitive} or {\it arc-transitive}. The {\it Euler characteristic} of $\calM$ is a topology parameter defined to be that of its supporting surface, that is,
\[\chi(\calM)=|V|-|E|+|F|.\] 

Edge-transitive maps are categorized into fourteen types according
to local structures and local actions in \cite{GW,ST}, among which five types are arc-transitive. The problem of constructing and classifying special classes of such maps with specific $\chi(\calM)$ has attracted considerable attention; refer to \cite{CNS,CPS,Char-p,bi-rotary} for $\chi(\calM)$ to be a negative prime or product of two primes. The following observation established in \cite[Lemma~2.2]{HLZZ} leads us to relaxing the restriction on $\chi(\calM)$ to being only square-free: 

{\it Let $\calM$ be a map and $ \Aut\calM $ be the automorphism group of $ \calM $. For any prime divisor $ p $ of $ |\Aut\calM| $, if $ p^2\nmid \chi(\calM) $, then each Sylow $ p $-subgroup of $ \Aut\calM $ has a cyclic or dihedral subgroup of index $ p $. In particular, if $ \chi(\calM) $ is square-free, then the group $ \Aut\calM $ satisfies {\rm Hypothesis {\rm \ref{hypo-0}}}.}\vs

The previous work \cite{HLZZ} characterizes the non-solvable groups which satisfy Hypothesis \ref{hypo-0}.
The current paper focuses on the $p$-groups with Hypothesis~\ref{hypo-0}, and a forthcoming work \cite{HLZZ2} will characterize the solvable groups in the general case. According to a classical result \cite[5.3.4]{Robinson}, $ p $-groups of odd orders that have a maximal cyclic subgroup are well-classified, with only three families: a cyclic family, an abelian family and a non-abelian meta-cyclic family. On the other hand, the automorphism groups of edge-transitive maps are inherently of even orders. Therefore, we only need to work on the case $ p=2 $.\vs

The main result of this paper is as follows, where as usual, $\ZZ_n$ denotes the cyclic group of order $n$, $\D_{2n}$, $\SD_{2n}$ and $\Q_{2n}$ respectively denote the dihedral, the \textit{semi-dihedral} and the \textit{generalized quaternion} group of order $2n$, and $A\circ B$ denotes the \textit{central product} of $ A $ and $ B $. Additional definitions are provided in Section 2.

\begin{theorem}\label{thm:p-gps}
Let $G$ be a finite $2$-group that has a cyclic or dihedral maximal subgroup.
Then $ G $ is one of the groups listed in the following tables, with $ \Aut(G) $ and some more information of $ G $ given.
\begin{table}[h]
	\newcommand{\tabincell}[2]{\begin{tabular}{@{}#1@{}}#2\end{tabular}}
	\centering
	\caption{\small The general case: $ \Aut(G) $ is a $ 2 $-group, $ \ell\ge2 $}\label{tab-main}
	\scalebox{0.8}{
		\begin{tabular}{ccccc}
			\toprule[1pt]
			$G$ & $ \Aut(G) $ & \tabincell{c}{\rm reversing \\ \rm triple} &  \tabincell{c}{\rm regular \\ \rm triple} & \tabincell{c}{\rm rotary \\ \rm pair}\\
			\midrule[1pt]
			
			$ \ZZ_{2^\ell} $ & $ \ZZ_{2^{\ell-2}}\times \ZZ_2 $ & $\times$ &  $\times$ & $ \surd $ \\
			\midrule[0.5pt]
			
			\tabincell{c}{$ \ZZ_4\times\ZZ_2 $ \\ $ \ZZ_{2^{\ell+1}}\times\ZZ_2$} & \tabincell{c}{$ \D_8 $ \\ $ (\ZZ_{2^{\ell-1}}\circ\D_8)\times\ZZ_2$} & $\times$ &  $\times$ & $ \surd $ \\
			\midrule[0.5pt]
			
			$ \D_{{2^{\ell+1}}} $ & $ \Hol(\ZZ_{2^\ell})$ & $\surd$ &  $\surd$ & $\surd$ \\
		    \midrule[0.5pt]
			
			$ \SD_{2^{\ell{+2}}}$ & $ \ZZ_{2^{\ell}}{:}\Aut(\ZZ_{2^{\ell+1}})$ & $\times $ &  $\times$ & $\surd $ \\
			\midrule[0.5pt]
			
			\tabincell{c}{$\ZZ_{2^{\ell+1}}{:}\ZZ_2=\l a\r{:}\l b\r$, \\ $ a^b=a^{2^{\ell}+1} $} 
			& $ (\ZZ_{2^{\ell-1}}\circ\D_8)\times\ZZ_2 $ & $\times $ &  $\times$ & $ \surd$ \\
			\midrule[0.5pt]
			
			$ \Q_{2^{\ell+2}}$
			& $ \Hol(\ZZ_{2^{\ell+1}}) $ & $\times $ &  $\times$ & $ \times $ \\
			\midrule[1pt]
			
			$\Q_{2^{\ell+2}}\circ \ZZ_4$ & $ \Aut(\Q_{2^{\ell+2}})\times\ZZ_2 $ & $\surd$ &  $\times$ & $ \times $ \\
			\midrule[0.5pt]
			
			$ \D_{2^{\ell+1}}\times\ZZ_2$
			& \tabincell{c}{$ H.(\D_8\times\ZZ_2) $, \\ $ H\cong \ZZ_{2^{\ell-1}}{:}\Aut(\ZZ_{2^\ell}) $}  & $\surd$ &  $\surd$ & $ \times $\\
			\midrule[0.5pt]
			
			\tabincell{c}{$ \ZZ_{2^{\ell+1}}{:}\ZZ_2^2=\l a\r{:}\l b,c\r$, \\ $ \l a,b\r=\D_{{2^{\ell+2}}}$, $ \l a, bc\r=\SD_{{2^{\ell+2}}}$}
			& $ (\Aut(\SD_{{2^{\ell+2}}})\times\ZZ_2){:}\ZZ_2 $ & $\surd$ &  $\surd$ & $ \times $\\		
			\bottomrule[1pt]
	\end{tabular}}
\end{table}
\begin{table}[h]
	\newcommand{\tabincell}[2]{\begin{tabular}{@{}#1@{}}#2\end{tabular}}
	\centering
	\caption{\small The degenerate case: $ \Aut(G) $ is not a $ 2 $-group }\label{tab-small}
	\scalebox{0.8}{
		\begin{tabular}{ccccc}
			\toprule[1pt]
			$G$ & $ \Aut(G) $ & \tabincell{c}{\rm reversing \\ \rm triple} &  \tabincell{c}{\rm regular \\ \rm triple} & \tabincell{c}{\rm rotary \\ \rm pair}\\
			\midrule[1pt]
			
			$ \ZZ_{2}^2=\D_4 $ & $ \S_3$ & $\surd$ & $\surd$  &  $\surd$ \\
			\midrule[0.5pt]
			
			$ \Q_{8}$ & $ \S_4$ & $\times$ &  $\times$ & $\times$ \\
			\midrule[1pt]
			
			$ \ZZ_{2}^3 $ & $ \GL_2(3)$ & $\surd$ & $\surd$  &  $\times$ \\
			\midrule[0.5pt]
			
			$\Q_{8}\circ \ZZ_4$ & $ \S_4\times\ZZ_2 $ & $\surd$ &  $\times$ & $ \times $\\
			\bottomrule[1pt]
	\end{tabular}}
\end{table}
\end{theorem}

The information in columns 3\,-\,5 of the tables contributes to the study of arc-transitive maps. Such maps are in fact determined by certain triples or pairs of elements within their automorphism groups; refer to Section 4 for further details. Recall \cite[Lemma~2.2]{HLZZ} (see page 1). Now as we focus on $ 2 $-groups, the restriction on $\chi(\calM)$ can be further relaxed to $ 4\nmid\chi(\calM) $. The second result of this paper is

\begin{theorem}\label{thm:maps-2-gps}
	Let $ G $ be a finite $ 2 $-group and $ \calM $ be a $ G $-arc-transitive map with $4\nmid\chi(\calM)$. Then the pair $ (G,\calM) $ is one of those pairs listed in {\rm Proposition \ref{2gp-rev-maps}, \ref{2gp-reg-maps}} or {\rm \ref{2gp-rot-maps}}. Further, $ \chi(\calM)=1 $, $ 2 $, $ 2-2^\ell $ or $2-2^{\ell}+2^{s} $ with $ \ell\ge s>1 $.
\end{theorem}

\noindent{\bf Remark.} It seems a non-trivial problem to identify integers $\ell\ge s>1 $ such that $ 2-2^{\ell} $ or $ 2-2^{\ell}+2^{s} $ is square free. All \textit{Mersenne numbers}, i.e., $ 2^p-1 $ with $ p $ a prime, are conjectured to be square-free. Meanwhile, there do exist some integers $ d $ such that $ 2^d-1 $ is divisible by a square $ x^2 $ as follows (see {\sf OEIS, A237043}).
\begin{center}
	\scalebox{0.8}{$\begin{array}{c|cccccc}
		\hline
		d & 6n & 20n & 21n & 110n & 136n & \cdots\\
		x & 3 & 5 & 7 & 11 & 17 &  \cdots\\
		\hline
	\end{array}$}
\end{center}

\section{$2$-Groups having a cyclic maximal subgroup} \label{sec:p-gps}

Denote by $\Q_{2^{\ell+1}}=\l a,c\r$ the {\it generalized quaternion group} of order $2^{\ell+1}$, where
\[{\mbox{$|a|=2^\ell\geqslant 2^2$, $a^{2^{\ell-1}}=c^2$,\ $a^c=a^{-1}$.}}\]
Then $c^2$ is the only involution, and $\Q_{2^{\ell+1}}=\ZZ_{2^\ell}.\ZZ_2$ is a non-split extension.

Denote by $ \SD_{2^{\ell{+1}}}=\l a,b\r=\l a\r{:}\l b\r $ the {\it semi-dihedral group} of order $2^{\ell+1}$, where
\[{|a|=2^\ell\geqslant 2^2,\ |b|=2,\ a^b=a^{2^{\ell-1}-1}.}\]
\hspace{0.7em} Let $A,B$ be finite groups such that the centers $\Z(A),\Z(B)$ have isomorphic subgroups $C_1,C_2$, respectively.
Let $\varphi$ be an isomorphism from $C_1$ to $C_2$, and let $C=\{(c,c^\varphi)\mid c\in C_1\}$.
Then $C$ is a normal subgroup of $A\times B$.
The factor group $(A\times B)/C$ is called a {\it central product} of $A$ and $B$, denoted by $A\circ_CB$,
and sometimes simply by $A\circ B$ if $C$ is equal to the center of one of the two groups.\vs

Let $ G $ be a $2$-group that has a maximal cyclic subgroup. The small-order case $ |G|\le 2^3 $ is easy to treat. We conclude that $ (G,\Aut(G)) $ is one of the following:
\[(\ZZ_4,\ZZ_2),\,(\ZZ_2^2,\S_3),\,(\ZZ_8,\ZZ_2^2),\,(\ZZ_4\times\ZZ_2,\D_8),\,(\D_8,\D_8),\,(\Q_8,\S_4).\]

\begin{proposition}\label{2-gp-cycle}
Let $G$ be a finite $2$-group of order $ |G|\ge 2^4 $ that has a cyclic maximal subgroup. 
Then $ G $ is one of the groups listed in the following table, with $\Aut(G) $ as shown. In particular, if $G$ is generated by involutions, then $ G $ is dihedral.

\begin{table}[h]
	\newcommand{\tabincell}[2]{\begin{tabular}{@{}#1@{}}#2\end{tabular}}
	\centering
	\caption{\small $ 2 $-groups having a cyclic maximal subgroup}\label{tab-cyclic}
	\scalebox{0.8}{
		\begin{tabular}{cccc}
			\toprule[1pt]
			$G$ & {\rm representation of $ G $} & {\rm a generating set of $ \Aut(G) $} & {\rm structure of $ \Aut(G) $}\\
			\midrule[1pt]
			
			$ \ZZ_{2^\ell} $ & $G=\l a\r $, $ a^{2^\ell}=1 $ & $ \Aut(G)=\l\rho_5,\rho_{-1}\r$, $a^{\rho_x}=a^x $ & $\ZZ_{2^{\ell-2}}\times\ZZ_2$\\
			\midrule[0.5pt]
			
			$ \ZZ_{2^\ell}\times\ZZ_2 $ & \tabincell{c}{$ G=\l a,b\r$, \\ $a^{2^\ell}=b^2=1 $} & \tabincell{c}{$ \Aut(G)=\l\rho_5,\rho_{-1},\tau,\s\r $, $ (a,b)^{\rho_{x}}=(a^x,b) $, \\ $ (a,b)^{\tau}=(ab,b) $, $ (a,b)^{\s}=(a,a^{2^{\ell-1}}b) $} & $ (\ZZ_{2^{\ell-2}}\circ\D_8)\times\ZZ_2 $ \\
			\midrule[0.5pt]
			
			$ \D_{2^{\ell{+1}}} $ & \tabincell{c}{$ G=\l a,b\r$, \\ $a^{2^\ell}=b^2=1, a^b=a^{-1} $} & \tabincell{c}{$ \Aut(G)=\l \rho_5,\rho_{-1},\eta\r$,\\ $(a,b)^{\rho_x}=(a^x,b)$, $(a,b)^\eta=(a,ab)$}  & $ \Hol(\ZZ_{2^{\ell}}) $ \\
			\midrule[0.5pt]
			
			$ \SD_{2^{\ell+1}} $ & \tabincell{c}{$ G=\l a,b\r$, \\ $a^{2^\ell}=b^2=1, a^b=a^{2^{\ell-1}-1} $} & \tabincell{c}{$ \Aut(G)=\l \rho_5,\rho_{-1},\eta_0\r$,\\ $(a,b)^{\rho_x}=(a^x,b)$, $(a,b)^{\eta_0}=(a,a^2b)$} & $ \ZZ_{2^{\ell-1}}{:}\Aut(\ZZ_{2^\ell}) $\\
			\midrule[0.5pt]
			
			$ \ZZ_{2^{\ell}}{:}\ZZ_2 $ & \tabincell{c}{$ G=\l a,b\r$, \\ $a^{2^\ell}=b^2=1, a^b=a^{2^{\ell-1}+1} $} & \tabincell{c}{$\Aut(G)=\l\rho_5,\rho_{-1},\tau,\s\r$, $ (a,b)^{\rho_{x}}=(a^x,b) $, \\ $ (a,b)^{\tau}=(ab,b) $, $ (a,b)^{\s}=(a,a^{2^{\ell-1}}b) $} & $ (\ZZ_{2^{\ell-2}}\circ\D_8)\times\ZZ_2 $ \\
			\midrule[0.5pt]
			
			$ \Q_{2^{\ell+1}} $ & \tabincell{c}{$ G=\l a,c\r$, \\ $a^{2^\ell}=1, a^{2^{\ell-1}}=c^2, a^c=a^{-1} $} & \tabincell{c}{$ \Aut(G)=\l \rho_5,\rho_{-1},\eta\r$,\\ $(a,c)^{\rho_x}=(a^x,c)$, $(a,c)^\eta=(a,ac)$} & $ \Hol(\ZZ_{2^{\ell}}) $ \\
			\bottomrule[1pt]
	\end{tabular}}
\end{table}
\end{proposition}

\begin{proof}
According to a classic result \cite[5.3.4]{Robinson}, finite $ 2 $-groups that have a cyclic maximal subgroup are precisely those listed in the table. Let $ G $ be one of the groups with its given representation.

The automorphism groups of cyclic or dihedral groups are well-known. For cyclic groups $ G=\ZZ_{2^{\ell}} $, we have $\Aut(G)=\{\rho_{i}\,|\,\rho_i:a\rightarrow a^i,\, 0\le i<2^{\ell},\, 2\nmid i\} $, and further, the structure of $ \Aut(G) $ is known as follows (see \cite[Theorem~5.44]{Rotman}): \[\Aut(G)=\l\rho_5\r\times\l\rho_{-1}\r\cong\ZZ_{2^{\ell-2}}\times\ZZ_2,\] where the subgroup $ \l \rho_5\r $ contains all $ \rho_i $ with $ i \equiv 1\, (\mod 4) $. For dihedral groups $ G=\D_{{2^{\ell+1}}} $, we have $ \Aut(G)= \{g_{i,j}\,|\,g_{i,j}:(a,b)\rightarrow(a^i,a^jb),\, 0\le i,j<2^{\ell},\,2\nmid i\} $, and it is easy to see $ \Aut(G)=\l g_{1,j}\r{:}\l g_{i,0}\r $. Let $ \eta=g_{1,1} $, $ \rho_{i}=g_{i,0} $. Then \[\Aut(G)=\l \eta\r{:}\l\rho_{5},\,\rho_{-1}\r\cong\ZZ_{2^\ell}{:}\Aut(\ZZ_{2^\ell})=\Hol(\ZZ_{2^\ell}).\]

The case $ G=\Q_{{2^{\ell+1}}} $ is similar as the dihedral case. Let $ g $ be an automorphism. The images $a^g$ and $c^g$ are of order $ 2^{\ell}$ and $ 4 $, respectively. Note that, $| a^\lambda c|=4 $ for $ 0\le\lambda <2^{\ell} $. Thus, $(a^g,c^g)=(a^i,a^jc) $, where $ 0\le i,j< 2^{\ell}$ and $ 2\nmid i $.

Suppose $ G=\SD_{{2^{\ell+1}}} $. Then the group $ \Aut(G) $ is given in \cite[Theorem~4.6 II]{2-gp-split} as $ \Aut(G)= \{g_{i,j}\,|\,g_{i,j}:(a,b)\rightarrow(a^i,a^{j}b),\,0\le i,j<2^{\ell},\,2\nmid i,\,2\mid j\} $, and then $ \Aut(G)=\l g_{1,j}\r{:}\l g_{i,0}\r $. Let $ \eta_0=g_{1,2} $, $ \rho_{i}=g_{i,0} $. Then \[\Aut(G)=\l\eta_0\r{:}\l \rho_{5},\rho_{-1}\r\cong\ZZ_{2^{\ell-1}}{:}\Aut(\ZZ_{2^\ell}).\] Note that, $ \Aut(G) $ is regarded as a normal subgroup of $ \Hol(\ZZ_{2^{\ell}}) $ of index $ 2 $.

The remaining two cases $ G= \ZZ_{2^\ell}\times\ZZ_2 $ and $ G=\ZZ_{2^\ell}{:}\ZZ_2 $ are similar to each other. The general structure of $ \Aut(G) $, where $ G $ is a direct product of groups, or $ G $ is a split metacyclic 2-group, is given in \cite{Aut-directproduct} or \cite{2-gp-split}, respectively. Here we provide a more refined description of the group $ \Aut(G) $.

Suppose (i) $ G=\ZZ_{2^\ell}\times\ZZ_2 $ or (ii) $ G=\ZZ_{2^\ell}{:}\ZZ_2 $, with the group representation as given in the table. Let $ g\in\Aut(G) $ and let $ a_0=a^{2^{\ell-1}} $. The images $ a^g,b^g $ satisfy:\[ \mbox{$|a^g|=2^\ell$, $|b^g|=2$ and $ (a^g)^{b^g}=a^g$ or $a_0a^g$,}\] respectively for (i) or (ii). For $ 0\le m<2^{\ell} $, $(a^mb)^2=a^m(a^m)^b=a^{2m} $ or $ a_0^ma^{2m} $, so that, $ |a^mb|=2^\ell$ if $ 2\nmid m $, and $ |a^mb|=2 $ if $ m=0 $ or $ 2^{\ell-1} $. Thus, the possible values for $ a^g,b^g $ are as follows: \[\mbox{$ a^g\in\{a^i, a^ib\,|\, 0\le i<2^{\ell},\, 2\nmid i\}$ and $ b^g\in\{b, a_0b \}$,}\] and $ a^g,b^g $ can surely take all such values, as for any odd $ i $, respectively we have \[\mbox{$ (a^i)^{a_0b}=(a^i)^b=a^i $ or $ a_0a^i $, and $ (a^ib)^{a_0b}=(a^ib)^b=a^ib $ or $ a_0a^ib $.}\] Since $ g $ is determined by the images $ a^g ,b^g$, we have $|\Aut(G)|=2\cdot2^{\ell-1}\cdot2=2^{\ell+1}$.

Now, define automorphisms $\rho_i\,(0\le i<2^{\ell}),\tau,\s $ as follows: \[\mbox{$ (a,b)^{\rho_i}=(a^i,b) $, $ (a,b)^\tau=(ab,b) $ and $ (a,b)^\s=(a,a_0b) $.}\] Let $ X=\l\rho_5,\rho_{-1},\tau,\s\r\le\Aut(G) $. Then $ \l\rho_5,\rho_{-1}\r=\l\rho_5\r\times\l \rho_{-1}\r\cong\Aut(\l a\r) $, and the following relations are satisfied:
\[(a,b)^{\tau\s\tau\s}=(a_0a,b)=(a,b)^{\rho_{1+2^{\ell-1}}},\,(a,b)^{\s\rho_i\s}=(a^i,b)=(a,b)^{\rho_i}.\] Particularly for case (i), $ a^b=b $, we have \[ (a,b)^{\tau\rho_i\tau}=((ab)^ib,b)=(a^i,b)=(a,b)^{\rho_i},\] so that, $ \l\rho_5\r $, $ \l\tau,\s\r $ and $\l\rho_{-1}\r$ are pair-wise commutative, with $ (\tau\s)^2=\rho_{1+2^{\ell-1}} $ being of order $ 2 $ and lying in $\l\rho_5\r\cap\l\tau,\s\r$. Thus, $ \l\tau,\s\r\cong \D_8 $, and 
\[X=(\l \rho_5\r\times\l \tau,\s\r/\l \rho_{1+2^{\ell-1}}(\tau\s)^2 \r)\times\l \rho_{-1}\r\cong (\ZZ_{2^{\ell-2}}\circ\D_8)\times\ZZ_2.\]
\noindent Meanwhile for case (ii), $ a^b=a_0b $, we have \[(a,b)^{\tau\rho_i\tau}=((ab)^ib,b)=(((ab)^2)^{\frac{i-1}{2}}a,b)=(a_0^{\frac{i-1}{2}}a^i,b)=
\begin{cases}
	(a,b)^{\rho_i},& i\equiv 1\, (\mod 4),\\
	(a,b)^{\rho_{i+2^{\ell-1}}},& i\equiv -1\, (\mod 4),\\
\end{cases}\] and further, \[(\rho_{-1}\s)^\tau=\rho_{-1}^\tau\s^\tau=\rho_{-1+2^{\ell-1}}.\rho_{1+2^{\ell-1}}\s=\rho_{-1}\s,\]so that, $ X=\l\rho_5,\rho_{-1},\tau,\s\r=\l\rho_5,\rho_{-1},\tau,\rho_{-1}\s \r, $ where $ \l\rho_5\r $, $ \l\rho_{-1},\tau\r $ and $\l\rho_{-1}\s \r$ are pair-wise commutative, with $ (\rho_{-1}\tau)^2=\rho_{-1}\rho_{-1}^\tau=\rho_{1+2^{\ell-1}} $ being of order $ 2 $ and lying in $\l\rho_5\r\cap\l\rho_{-1},\tau\r$. Thus, $ \l\rho_{-1},\tau\r\cong\D_8 $, and
\[X=(\l \rho_5\r\times\l \rho_{-1},\tau\r/\l \rho_{1+2^{\ell-1}}(\rho_{-1}\tau)^2\r)\times\l \rho_{-1}\s\r\cong (\ZZ_{2^{\ell-2}}\circ\D_8)\times\ZZ_2.\]

Finally, since $ |X|=2^{\ell+1}=|\Aut(G)| $, we have $ X=\Aut(G) $.\vs

At last, assume that $G$ is generated by involutions. Then $ G\ne \ZZ_{{2^\ell}}$, $\ZZ_{{2^\ell}}\times\ZZ_2$, $ \Q_{{2^{\ell+1}}} $. Suppose $ G=\SD_{2^{\ell{+1}}}=\l a\r{:}\l b\r $ with $a^b=a^{2^{\ell-1}-1}$. Then $(a^ib)^2=a^i(a^i)^b=a^ia^{i(2^{\ell-1}-1)}=a^{i(2^{\ell-1})}$, and so $ a^ib $ is an involution if and only if $ i $ is even. Thus, all involutions lie in the subgroup $\l a^2,b\r$, a contradiction.
Suppose $G=\l a\r{:}\l b\r=\ZZ_{2^\ell}{:}\ZZ_2$ with $a^b=a^{2^{\ell-1}+1}$. Then it is shown that all involutions lie in $\l a^{2^{\ell-1}},b\r$. At last, dihedral groups $ \D_{{2^{\ell+1}}} $ can be generated by involutions.
\end{proof}

\section{2-Groups having a dihedral maximal subgroup}

In this section, we first investigate three special families of $ 2 $-groups and then prove that these are all possible groups obtainable in this case.

The first family is given below; refer to \cite[Definition 4.2 \& Lemma~4.3]{HLZZ}.

\begin{definition}\label{defi-Q}
	{\rm
	Let $\Q_{2^{\ell+1}}=\l a,c\,|\,a^{2^\ell}=1,\,a^{2^{\ell-1}}=c^2,\,a^c=a^{-1}\r$, $\ell\ge2$. Let
	\[G=\Q_{2^{\ell+1}}\circ\ZZ_4=\l a,c\r\circ \l d\r=(\l a,c\r\times \l d\r)/\l c^2d^2\r.\]
	Note that, the only involutions contained in $\l a,c\r$ and $\l d\r$ are $a^{2^{\ell-1}}=c^2$ and $d^2$, respectively, which are actually identical in $ G $ as $c^2d^2=1$.}
\end{definition}

\begin{lemma}\label{Z-circ-Q-hypo-2} 
	Let $G=\Q_{2^{\ell+1}}\circ\ZZ_4=\l a,c\r\circ \l d\r$, defined in {\rm Definition \ref{defi-Q}}. Then
	\begin{enumerate}[{\rm(1)}]
		\item $\l a,cd\r=\l a\r{:}\l cd\r\cong\D_{2^{\ell+1}}$ is a dihedral subgroup of $ G $ of index $2${\rm;}\vs
		
		\item $G$ is not generated by two elements{\rm;}\vs
		
		\item the involutions of $ G $ are $a^{2^{\ell-1}}$, $a^{2^{\ell-2}}d^{\pm 1}$ and $a^icd$, where $0\leqslant i< 2^{\ell}${\rm;}\vs
		
		\item if $x,y,z$ are involutions generating $G$, then $\{x,y,z\}=\{ a^{2^{\ell-2}}d,a^icd, a^jcd\}$ or $ \{ a^{2^{\ell-2}}d^{-1},a^icd, a^jcd\} $,
		where $0\leqslant i,j< 2^{\ell}$ and ``$ i-j $ is odd''; in particular, the three involutions are pairwise non-commutative.
	\end{enumerate}
\end{lemma}

\noindent{\bf Remark.} The condition {``$ i-j $ \textit{is odd}"} is necessary, yet \cite[Lemma~4.3]{HLZZ} missed this.

\begin{lemma}\label{ZQ-aut}
    Let $G=\Q_{2^{\ell+1}}\circ \ZZ_4=\l a,c\r\circ \l d\r$, defined in {\rm Definition \ref{defi-Q}}. Then $ \Aut(G)=\Aut(\l a,c\r)\times\l\tau\r\cong\Aut(\Q_{{2^{\ell+1}}})\times\ZZ_2$, where $ \tau $ is with $ (a,c,d)^{\tau}=(a,c,d^{-1}) $.
\end{lemma}

\begin{proof} By straight calculation we have: \begin{itemize}
		\item[(i)] if $ \ell= 2 $, then $ \l a\r $, $ \l c\r $, $ \l ac\r $, $ \l d\r $ and $ \l a^2d\r $ are the only subgroups of $ G $ of order $ 4 $, where $ \l d\r $ and $ \l a^2d\r $ lie in the center $ \Z(G) $;\vs
		\item[(ii)] if $ \ell\ge 3 $, then $ \l a\r $, $ \l ad\r $ are the only two subgroups of $ G $ of order $ 2^\ell $, while the elements of order $ 4 $ are of form $ a^\lambda $, $ a^\lambda d$ or $ a^ic $, with certain $ \lambda $ and $ 0\le i<2^\ell  $.
	\end{itemize}
	It then follows in both cases that $ \l a,c\r $ is the only subgroup of $ G $ isomorphic to $ \Q_{2^{\ell+1}} $, and so it is characteristic. Let $ g\in\Aut(G) $. Then the images $ a^g,c^g\in \l a,c\r $, and meanwhile $ d^g=d $ or $ d^{-1} $, the only two central elements of order $ 4 $.
\end{proof}

The second family is given below.

\begin{definition}\label{defi-direct}
	{\rm
	Let $ G=\D_{2^{\ell+1}}\times\ZZ_2=\l a,b\r\times\l c\r$, $ \ell\ge 2 $, with representation \[G=\l a,b,c\,|\,a^{2^\ell}=b^2=c^2=1,\, a^b=a^{-1},\,a^c=a,\,b^c=b\r.\]
    }
\end{definition}

\begin{lemma}\label{Klein-I}
	Let $ G=\D_{2^{\ell+1}}\times\ZZ_2=\l a, b\r\times\l c\r$, defined in {\rm Definition \ref{defi-direct}}. Then
	\begin{itemize}
		\item[\rm (1)] $ \l a^2\r \lhd G $, $ G/\l a^2\r\cong\ZZ_2^3 $, and $G$ is not generated by two elements{\rm;}\vs
		\item[\rm (2)] the set of involutions $ \Omega=\{a_0,c,a_0c,a^{i}b,a^{i}bc\,|\,0\le i< 2^{\ell}\} $, where $ a_0=a^{2^{\ell-1}} $.
	\end{itemize}
\end{lemma}

\begin{proof}
	(1) As $ (a^2)^b=(a^2)^{-1} $, $ (a^2)^c=a^2 $, we have $ \l a^2\r \lhd G $, and so $ \ov G:=G/\l a^2\r=\l \ov a,\ov b,\ov c \r\cong\ZZ_2^3 $. Suppose $ G=\l x,y\r $. Then $ \ov G $ is generated by $ \ov x, \ov y $, which is impossible.
	
	(2) Let $ \Omega(X) $ be the set of involutions of $ X\leqslant G $, and let $ a_0=a^{2^{\ell-1}} $. Then $\Omega(\l a,b\r)=\{a_0,a^{i}b\,|\,0\le i< 2^{\ell}\}$ as $ \l a,b\r\cong\D_{{2^{\ell+1}}} $, and $ \Omega(\l c\r)=\{c\} $. So, $ \Omega=\{a_0,a^{i}b,c,a_0c,a^{i}bc\,|\,0\le i< 2^{\ell}\} $.
\end{proof}

\begin{lemma}\label{2-gp-aut-I}
	Let $ G=\D_{2^{\ell+1}}\times\ZZ_2=\l a, b\r\times\l c\r$, defined in {\rm Definition \ref{defi-direct}}. Then $ \Aut(G)=\l\rho_{5},\rho_{-1}, \eta,\tau,\s\r$ is generated by the following automorphisms: \[\mbox{$ (a,b,c)^{\rho_{5}}=(a^5,b,c) $, $ (a,b,c)^{\rho_{-1}}=(a^{-1},b,c) $,\,\, $ (a,b,c)^{\eta}=(a,ab,c) $,}\] \[\mbox{$(a,b,c)^\tau=(ac,b,c)$,\, $(a,b,c)^\s=(a,b,a^{2^{\ell-1}}c)$.}\] Furthermore, letting $ H=\l\rho_{5},\rho_{-1}, \eta^2\r$, then $H\cong \ZZ_{2^{\ell-1}}{:}\Aut(\ZZ_{2^\ell})$ is normal in $\Aut(G)$, while the factor group $\Aut(G)/H$ is isomorphic to $\D_8\times\ZZ_2 $. 
\end{lemma}

\begin{proof}
	The general structure of $ \Aut(G) $, where $ G $ is a direct product of groups, is given in \cite{Aut-directproduct}. Here we provide a more refined description of the group $ \Aut(G) $.
	
	Let $ a_0=a^{2^{\ell-1}} $. Then $\Z(\l a,b\r)=\l a_0\r $. Let $ g\in\Aut(G) $. By \cite[Theorem~3.2]{Aut-directproduct}, the images $ a^g,b^g,c^g $ take values as follows:
	\[ a^g=a^\a a^\gamma,\,b^g=b^\a b^\gamma,\,c^g=c^\b c^\delta,\mbox{\ where}\]  
	\begin{itemize}
	    \item[(i)] $ \a\in\Aut(\l a,b\r)=\Aut(\D_{{2^{\ell+1}}}) $, so that, $ a^\a=a^i $, $ b^g=a^jb $, $ 0\le i,j<2^{\ell} $, $ 2\nmid i $;\vs
	    \item[(ii)] $ \gamma\in\Hom(\l a,b\r,\Z(\l c\r))=\Hom(\D_{{2^{\ell+1}}},\ZZ_2) $, so that, $ a^\gamma $, $ b^\gamma=1 $ or $ c $;\vs
	    \item[(iii)] $ \b\in\Aut(\l c\r)=\Aut(\ZZ_2) $ and $\delta\in\Hom(\l c\r,\Z(\l a,b\r))=\Hom(\l c\r,\l a_0\r) $, so that, $ c^\b=c $ and $ c^\delta= 1$ or $a_0 $;
	\end{itemize} 
    that is, \[ a^g\in\{a^i,\,a^ic\},\, b^g\in\{a^jb,\, a^jbc\},\, c^g\in\{c,\,a_0c\}, \] where $ 0\le i,j<2^{\ell} $ and $ 2\nmid i$. Since $ g $ is determined by the images $ a^g ,b^g, c^g$, we have \[|\Aut(G)|=2^3\cdot2^{\ell-1}\cdot2^{\ell}=2^{2\ell+2}.\]
	
	Now, define automorphisms $ \rho_{i}\,(0\le i<2^{\ell}, 2\nmid i),\eta_0,\tau,\s $ as follows:  \[\mbox{$ (a,b,c)^{\rho_{i}}=(a^i,b,c) $,\, $ (a,b,c)^{\eta}=(a,ab,c) $,}\] \[\mbox{$(a,b,c)^\tau=(ac,b,c)$,\, $(a,b,c)^\s=(a,b,a_0c)$.}\]
    Let $ X,Y $ be subgroups of $ \Aut(G) $ with $ X=\l\rho_{5},\rho_{-1}, \eta,\tau,\s\r$, $ Y=\l \rho_5,\rho_{-1},\eta\r$. Then $ Y=\l\eta\r{:}\l\rho_5,\rho_{-1}\r\cong\Aut(\l a,b\r)\cong\Aut(\D_{{2^{\ell+1}}}) $ has a normal subgroup $ H $ of index $ 2 $ as \[H=\l\eta^2\r{:}\l\rho_5,\rho_{-1}\r\cong \ZZ_{2^{\ell-1}}:\Aut(\ZZ_{2^\ell}). \] It is straight to check the following relations: \[ (a,b,c)^{\s\rho_{i}\s\rho_{i}^{-1}}=(a,b,c)^{\s\eta\s\eta^{-1}}=(a,b,c)^{\tau\rho_{i}\tau\rho_{i}^{-1}}=(a,b,c),\] \[\ (a,b,c)^{\tau\eta\tau\eta^{-1}}=(a,bc,c),\ (a,b,c)^{\tau\eta^{2}\tau\eta^{-2}}=(a,b,c),\mbox{\ and\ }\] \[(a,b,c)^{\tau\s\tau\s}=(a_0a,b,c)=(a,b,c)^{\rho_{1+2^{\ell-1}}}\in(a,b,c)^{H}.\] 
	Therefore, $H\lhd Y $ and $ H $ commutes with $ \l\s,\tau\r $, so $ H \lhd X $. Let $ \ov X=X/H $. Then $ \ov X=\l \ov{\eta}, \ov\tau,\ov\s\r\cong \D_8\times\ZZ_2 $, because
	\begin{itemize}
	\item[(a)] $|\ov {\eta}|=|\ov\tau|=|\ov\s|=2$;
	\item[(b)] $ (a,b,c)^{\eta\tau}=(ac,abc,c) $, $ (a,b,c)^{(\eta\tau)^4}=(a,a^4b,c)=(a,b,c)^{\eta^4} $ where $ \eta^4\in H $, so $ |\ov{\eta}\,\ov\tau|=4 $, $ \l \ov\eta,\ov\tau\r\cong\D_8 $;
	\item[(c)] $ \l \ov\eta,\ov\tau\r$ commutes with $ \l\ov\s\r$, as $[\eta,\s]=1$, $ [\tau,\s]\in H $.
	\end{itemize}
	Finally, since $ |X|=2^{2\ell+2}=|\Aut(G)| $, we have $ X=\Aut(G) $.
\end{proof}

The third family is given below. The groups belonging to this family are in fact double covers of $ \D_{{2^{\ell}}}\times\ZZ_2 $. Let $ X,Y $ be two finite groups, and denote respectively by $ \Z(X) $ and $ X' $ the \textit{center} and the \textit{commutator subgroup} of $ X $. We say that $ X $ is a \textit{covering group} of $ Y $ if $ \Z(X)\le X' $ and $ X/\Z(X)\cong Y $. If the center has order $ 2 $, the covering group is often referred to as a \textit{double cover}.

\begin{definition}\label{defi-semidirect}
	{\rm
		Let $ G=\D_{2^{\ell+1}}{:}\ZZ_2=\l a,b\r{:}\l c\r$, $ \ell\ge 3 $, with representation \[G=\l a,b,c\,|\,a^{2^\ell}=b^2=c^2=1,\, a^b=a^{-1},\,a^c=a^{2^{\ell-1}+1},\,b^c=b\r.\] Also, $ G $ can be rewritten as $ G=\l a\r{:}\l b,c\r=\ZZ_{2^{\ell}}{:}\ZZ_2^2 $. Note that, $ b,c,bc $ are exactly the only three involutions of $ \Aut(\l a\r)\cong\ZZ_{2^{\ell-2}}\times \ZZ_2 $, and further,
		\[\l a,b\r\cong\D_{{2^{\ell+1}}},\,\l a,bc\r\cong\SD_{{2^{\ell+1}}},\, \l a,c\r=\ZZ_{2^{\ell}}{:}\ZZ_2 \mbox{\ with\ } a^c=a^{2^{\ell-1}+1}.\]
	}
\end{definition}

\begin{lemma}\label{Klein-II}
	Let $ G=\D_{2^{\ell+1}}{:}\ZZ_2=\l a, b\r{:}\l c\r$, defined in {\rm Definition \ref{defi-semidirect}}. Let $ a_0=a^{2^{\ell-1}} $. Then the following statements hold.
	\begin{itemize}
		\item[\rm (1)] $ \l a^2\r \lhd G $, $ G/\l a^2\r\cong\ZZ_2^3 $, and $G$ is not generated by two elements.\vs
		\item[\rm (2)] The set of involutions $ \Omega=\{a_0,c,a_0c,a^{i}b,a^{2i}bc\,|\,0\le i<2^{\ell}\} $.\vs
		\item[\rm (3)] $ G $ is a double cover of $ \D_{{2^{\ell}}}\times \ZZ_2 $, where $ \Z(G)=\l a_0\r $.
		%\item[\rm (3)] if $ u,v,w $ are involutions such that $ G=\l u,v,w\r $, then $a_0\notin\{u,v,w\}${\rm;}\vs
	\end{itemize}
\end{lemma}

\begin{proof}
	(1) As $ (a^2)^b=(a^2)^{-1} $, $ (a^2)^c=a^2 $, we have $ \l a^2\r \lhd G $, and so $ \ov G:=G/\l a^2\r=\l \ov a,\ov b,\ov c \r\cong\ZZ_2^3 $. Suppose $ G=\l x,y\r $. Then $ \ov G $ is generated by $ \ov x, \ov y $, which is impossible.
	 
	(2) Let $ L=\l a,b\r $, $ M=\l a,bc\r $, $ N=\l a,c\r $. Note that $ G=L\cup M\cup N $, so letting $ \Omega(X) $ be the set of involutions of $ X\leqslant G $, then $\Omega(G)=\Omega(L)\cup\Omega(M)\cup\Omega(N)$. Now, $ L\cong \D_{{2^{\ell+1}}} $, $ M\cong\SD_{2^{\ell{+1}}} $ and $ N=\l a,c\r=\ZZ_{2^{\ell}}{:}\ZZ_2 $ with $ a^c=a^{2^{\ell-1}+1} $, so $\Omega(L)=\{a_0,a^xb\,|\,0\le x<2^{\ell}\}$, and $ \Omega(M) $, $ \Omega(N) $ are already given in the proof of Proposition \ref{2-gp-cycle} as $\Omega(M)=\{a_0,a^{2x}bc\,|\,0\le x<2^{\ell-1}\}$, $\Omega(N)=\{a_0,c,a_0c\}$.
	
	(3) Let $ L=\l a,b\r\cong \D_{{2^{\ell+1}}} $. It is clear that $ \Z(L)=\l a_0\r\cong\ZZ_2$, $ L'=\l a^2\r\cong\ZZ_{2^{\ell-1}} $. Note that $a ^c, a^{bc}\ne a $, so elements of form $ a^{\lambda}c $ or $ a^{\lambda}bc $ do not commute with $ a $. Thus, $ \Z(G)\le\Z(L) $, and so $ \Z(G)=\Z(L) $ as $ (a_0)^c=(a_0)^{2^{\ell-1}+1}=a_0 $. Then we have $ \Z(G)\le L'\le G' $. Let $ \ov G=G/\Z(G) $. As $ a^c=a^{2^{\ell-1}+1}=a_0a $, $ \ov a^{\ov c}={\ov {a^c}}=\ov a $, and so $\ov G=\l\ov a,\ov b\r\times\l \ov c\r\cong \D_{{2^{\ell}}}\times \ZZ_2 $.
\end{proof}

\begin{lemma}\label{2-gp-aut-II}
	
	Let $ G=\D_{2^{\ell+1}}{:}\ZZ_2=\l a, b\r{:}\l c\r$, defined in {\rm Definition \ref{defi-semidirect}}. Then $ \Aut(G)=\l\rho_{5},\rho_{-1}, \eta_0,\tau,\s\r$ is generated by the following automorphisms: \[\mbox{$ (a,b,c)^{\rho_{5}}=(a^5,b,c) $,\, $ (a,b,c)^{\rho_{-1}}=(a^{-1},b,c) $,\, $ (a,b,c)^{\eta_0}=(a,a^2b,c) $,}\] \[\mbox{$(a,b,c)^\tau=(ac,bc,c)$,\, $(a,b,c)^\s=(a,b,a^{2^{\ell-1}}c)$.}\] Furthermore, letting $ H=\l\rho_{5},\rho_{-1},\eta_0\r$, then $H\cong\Aut(\l a,bc\r)\cong\Aut(\SD_{{2^{\ell+1}}})$, and  $\Aut(G)=(H\times\l\s\r){:}\l\tau\r\cong((\ZZ_{2^{\ell-1}}{:}\Aut(\ZZ_{2^\ell}))\times\ZZ_2){:}\ZZ_2$.
\end{lemma}

\begin{proof}
	As $ G=\l a,b,c\r $, any given $ g\in\Aut(G) $ is determined by the images $ a^g ,b^g, c^g$. Let $ a_0=a^{2^{\ell-1}} $. According to Lemma \ref{Klein-II}, $ \Z(G)=\l a_0\r $, the set of involutions \[ \Omega=\{a_0,\,c,\,a_0c,\,a^{\lambda}b,\,a^{2\lambda}bc\,|\,0\le \lambda<2^{\ell}\},\] and the factor group $ \ov G=G/\Z(G)=\l\ov a,\ov b\r\times\l\ov c\r\cong\D_{{2^{\ell}}}\times\ZZ_2 $. Now, $ g $ induces an automorphism $ g_0 $ of $ \ov G $, which by Lemma \ref{2-gp-aut-I} maps $ \ov c $ to $ \ov c $ or $ \ov {a^{2^{\ell-2}}c} $. Correspondingly, $ g $ maps $ c $ to one of $ \{c,\, a_0c,\,a^{2^{\ell-2}}c,\,a_0a^{2^{\ell-2}}c\} $. Only the first two elements are involutions. Thus, $ c^g\in\{c,a_0c\} $.
	
	Meanwhile, note that $ b^g,(ab)^g\in\Omega $ are involutions such that $ [b^g,c^g]=1 $ and $ G=\l(ab)^g,b^g,c^g\r $. If $ b^g=a_0c^g $, then $ \ov {b^g}=\ov {c^g} $, so that $ \ov G=\l \ov {a^g},\ov {b^g},\ov {c^g}\r=\l \ov {a^g},\ov {b^g}\r $ is generated by $ 2 $ elements, which is impossible as $ \ov G\cong\D_{{2^{\ell}}}\times\ZZ_2 $. Thus, $b^g\in\{a^\lambda b,\, a^{2\lambda}bc\,|\,0\le i<2^{\ell}\} $, and since $ [a^2,c]=[b,c]=1 $, $ [a,c]=a_0\ne1 $, the possible values such that $ [b^g,c^g]=1 $ are as follows: \[\mbox{$ b^g=a^{2j}b $ or $ a^{2j}bc $, $ 0\le j<2^{\ell-1} $.}\] It follows that the possible values such that $ G=\l(ab)^g,b^g,c^g\r $ are $ (ab)^g=a^kb $, with $ 0\le k<2^{\ell} $, $ k $ odd, and then respectively for $ b^g=a^{2j}b $ or $ a^{2j}bc $,  we have \[\mbox{$ a^g=(ab)^gb^g=a^{i} $ or $a^ic$, where $i=k-2j$, so $ 0\le i<2^{\ell} $, $ 2\nmid i $.}\] 
	
	Note that, $a^g ,b^g, c^g$ can surely take all the values given as above, as for any odd $ i $, the following relations are always satisfied: \[\mbox{$ |a^{i}|=2^{\ell}$, $|a^{i}c|=2^{\ell} $ since $(a^ic)^2=a^i(a^i)^c=a_0a^{2i}$, and}\] \[\mbox{$ (a^i)^{a_0c}=(a^i)^{c}=a_0^ia^i=a_0a^i $, $ (a^ic)^{a_0c}=(a^ic)^{c}=a_0^ia^ic=a_0a^ic $.}\] Therefore, we have $ |\Aut(G)|=2\cdot2\cdot2^{\ell-1}\cdot2^{\ell-1}=2^{2\ell}$.
    
   Now, define automorphisms $ \rho_{i}\,(0\le i<2^{\ell}, 2\nmid i),\eta_0,\tau,\s $ as follows: \[(a,b,c)^{\rho_{i}}=(a^i,b,c),\,(a,b,c)^{\eta_0}=(a,a^2b,c),\]
   \[(a,b,c)^\tau=(ac,bc,c),\, (a,b,c)^\s=(a,b,a_0c).\] Let $ X,H $ be subgroups of $ \Aut(G) $ with $ X=\l\rho_{5},\rho_{-1},\eta_0,\tau,\s\r $, $ H=\l\rho_{5},\rho_{-1},\eta_0\r $. Note that, $ \rho_{i} $ maps $ (a,b,c) $ to $ (a^i,b,c) $, so maps $ (a,bc,c) $ to $ (a^i,bc,c) $; $ \eta_0 $ maps $ (a,b,c) $ to $ (a,a^2b,c) $, so maps $ (a,bc,c) $ to $ (a,a^2bc,c) $. Thus by Proposition \ref{2-gp-cycle}, $H\cong\Aut(\l a,bc\r)\cong\Aut(\SD_{{2^{\ell+1}}})$. 
   
   It is straight to check the following relations: \[ (a,b,c)^{\s\rho_{i}\s\rho_{i}^{-1}}=(a,b,c)^{\s\eta_0\s\eta_0^{-1}}=(a,b,c),\] \[(a,b,c)^{\tau\rho_{i}\tau}=(a_0^{(i-1)/2}a^i,b,c)\in (a,b,c)^H,\ (a,b,c)^{\tau\eta_0\tau}=(a,a_0a^2b,c)\in (a,b,c)^H\mbox{\ and}\]
	\[(a,b,c)^{\tau\s\tau\s}=(a_0a,a_0b,c)\in(a,b,c)^{H}.\] Therefore, $ H $ commutes with $ \l\s\r $, and meanwhile normalized by $ \l\tau\r $. Then $ H $ is normal in $ X $, so is $\l H,\s\r=H\times\l\s\r $. Thus, we have \[ X=(H\times\l\s\r){:}\l\tau\r\cong((\ZZ_{2^{\ell-1}}{:}\Aut(\ZZ_{2^\ell}))\times\ZZ_2){:}\ZZ_2.\]
	
	Finally, since $ |X|=2^{2\ell}=|\Aut(G)| $, we have $ X=\Aut(G) $.
\end{proof}

\noindent\textbf{Remark.} Let $ D=\l a,b\r $ be a dihedral maximal subgroup of $ G $. Then the non-central involutions of $ D $ fall into two conjugacy classes: $ \mathcal{C}_1=\{a^ib\,|\,0\le i<2^{\ell},\,2\,|\,i\} $ and $ \mathcal{C}_2=\{a^ib\,|\,0\le i<2^{\ell},\,2\nmid i\}$. Meanwhile, let $ X=\{c,a_0c\} $. Then $ X $ commutes with $ \mathcal{C}_1 $, but not with $ \mathcal{C}_2 $ since $ [X,\mathcal{C}_2]=a_0\ne 1 $. Note that $ X $ is set-wisely fixed by $ \Aut(G) $; consequently, $ \Aut(G) $ cannot mix $ \mathcal{C}_1 $, $ \mathcal{C}_2$.\vs

The final result of this section indicates that a $ 2 $-group having a cyclic or dihedral maximal subgroup if and only if it is one of the groups we have previously described, with only one exception of small order, i.e., $ G=\ZZ_2^3 $.

\begin{proposition}\label{2-gp-dihedral}
Let $G$ be a finite $2$-group that has a dihedral maximal subgroup.
Assume that $G$ does not have a cyclic maximal subgroup.
Then, either $G=\ZZ_2^3$, or $ G $ is one of the following groups:
\begin{enumerate}[\rm(1)]
    %\item $\ZZ_2^3$.\vs
    \item $\Q_{2^{\ell+1}}\circ \ZZ_4 $, $ \ell\geqslant 2 $.\vs
    \item $ \D_{2^{\ell+1}}\times\ZZ_2 $, $ \ell\geqslant 2 $.\vs
    \item $\D_{2^{\ell+1}}{:}\ZZ_2=\l a,b\r{:}\l c\r$, where $\l a\r{:}\l b\r=\D_{2^{\ell+1}}$ and $(a,b)^c=(a^{2^{\ell-1}+1},b)$, $ \ell\ge 3 $.
\end{enumerate}
\end{proposition}

\begin{proof}
First, the $ 2 $-groups appearing in Proposition \ref{2-gp-cycle} (hence having a cyclic maximal subgroup) are all excluded. Particularly for small-order case $|G|\le 2^3$, the only remainder is $G=\ZZ_2^3$, which has a dihedral subgroup $ \ZZ_2^2=\D_4 $ of index $ 2 $.

Assume thus $|G|=2^{\ell+2}\ge 2^4$.
Let $D$ be a dihedral maximal subgroup of $ G $, which is of index $ 2 $, and so $ D\cong\D_{2^{\ell+1}} $. Suppose $ D=\l a,b\r $ with $|a|=2^\ell$, $|b|=2$, $a^b=a^{-1}$.
Note that $\l a\r$, as the only cyclic subgroup of $D$ of order $2^\ell$, is normal in $ G $.

Let $c$ be an element of $G\setminus D$ and let $H=\l a,c\r$. Then $c^2\in D$, $a^c=a^\lambda$ for some integer $\lambda$, and there are three possibilities for $c^2$:
\[\mbox{$c^2=a^i\not=1$, $c^2=a^ib$ or $c^2=1$, where $0\le i<2^{\ell}$.}\]
If $c^2=a^ib$, then $a^{\lambda^2}=a^{c^2}=a^{a^ib}=a^{-1}$, and so $\lambda^2\equiv -1$ $(\mod 2^\ell)$, which is impossible as $ \ell\ge 2 $.
Thus there are two cases to treat, namely, $c^2=a^i$, or $c^2=1$.

(i) Assume first $c^2=a^i\not=1$.
Then $H=\l a,c\r$ is of order $2^{\ell+1}$ and contains a cyclic subgroup $\l a\r$ of index $ 2 $.
By our assumption, $H$ itself can not be cyclic. Thus, by Proposition~\ref{2-gp-cycle}, the only possibility is \[\mbox{$H=\l a,c\r=\Q_{2^{\ell+1}}$,
so that $c^2=a^{2^{\ell-1}}$, $ a^c=a^{-1} $.}\]
Meanwhile, $b^c$ is a non-central involution of $D$, so $b^c=a^jb$, with $0\leqslant j< 2^\ell$.

Suppose $j=2m$. As $ (a^mb)^c=a^{-m}a^{2m}b=a^mb$, letting $b'=a^mbc$, then \[\mbox{$a^{b'}=a^{bc}=a$, $c^{b'}=c^{(a^mb)c}=c$, and $ (b')^2=(a^mb)^2c^2=c^2 $, so}\]
\[\mbox{$ G=\l a,b,c\r=\l a,c,b'\r=\l a,c\r\circ\l b'\r=\Q_{2^{\ell+1}}\circ\ZZ_4 $, as in part~(1).}\] 

Suppose $j=2m+1$. As $ (a^mb)^c=a^{-m}a^{2m+1}b=a^{m+1}b$, letting $b'=a^mbc$, then \[(b')^2=a^mbc^2(a^mb)^c=c^2a^mba^{m+1}b=a^{2^{\ell-1}-1},\] which is of order $ 2^\ell $. Thus, $ \l b'\r\cong\ZZ_{2^{\ell+1}} $ is a cyclic subgroup of $ G $ of index $ 2 $, which contradicts our assumption. In fact, $ G=\l a,b,c\r=\l b',b\r=\l b'\r{:}\l b\r=\SD_{2^{\ell{+2}}} $.\vs

(ii) Assume then $c^2=1$, so that $G=D{:}\l c\r\cong\D_{{2^{\ell+1}}}{:}\ZZ_2$ is a split extension. Thus, either $ G\cong\D_{{2^{\ell+1}}}\times\ZZ_2 $, or $ G\cong\D_{{2^{\ell+1}}}{:}\ZZ_2 $ is determined by the non-conjugate involutions of $ \Aut(\D_{{2^{\ell+1}}}) $. Explicitly, suppose $ a^c=a^\lambda $ and $ b^c=a^jb $, where $ \lambda,j $ are integers within $0\leqslant\lambda,j<2^\ell$ satisfying \[\mbox{$a^{c^2}=a^{\lambda^2}=a$, and $b^{c^2}=(a^jb)^c=a^{j\lambda}a^jb=a^{(\lambda+1)j}b=b$.}\]
Since $|a|=2^\ell$, it yields:

\noindent (a) $\lambda^2\equiv 1$ $(\mod 2^\ell)$, so that, $\lambda\in\{\pm 1\}$ for $ \ell=2 $, and $\lambda\in\{\pm 1,2^{\ell-1}\pm 1\}$ for $ \ell>2 $; 

\noindent (b) $(\lambda+1)j\equiv 0$ $(\mod 2^\ell)$.

\noindent We then analyze all the possibilities.

(ii.1) Assume $\lambda=1$. Then $2j\equiv0$ $(\mod 2^\ell)$. Thus, either $j=0$, $(a,b)^c=(a,b)$, \[\mbox{$G=\l a,b\r{:}\l c\r=\D_{2^{\ell+1}}\times\ZZ_2$, as in part~(2);}\] or $j=2^{\ell-1}$, $(a,b)^c=(a,a^{2^{\ell-1}}b)$. For the latter case, let $b'=abc$, $c'=a^{2^{\ell-2}}c$. As $ a,c $ are commutative, so are $ a,c' $ and $ c,c' $. Meanwhile, $ b,c' $ are commutative as $b^{c'}=b^{a^{2^{\ell-2}}c}=(a^{2^{\ell-1}}b)^{c}=a^{2^{\ell-1}}a^{2^{\ell-1}}b=b$. Then $G=\l a,b,c\r=\l a,b',c'\r$, where $ \l a,b'\r $ commutes with $ \l c'\r $. Further, \[\mbox{$ a^{b'}=a^{abc}=a^{-1} $, and $(b')^2=ab(ab)^c=aba^{1+2^{\ell-1}}b=a^{2^{\ell-1}}=(c')^2$.}\] Thus, $\l a,b'\r\cong \Q_{2^{\ell+1}}$, $\l c'\r\cong\ZZ_4$, and $G\cong \Q_{2^{\ell+1}}\circ\ZZ_4$, as in part~(1).

(ii.2) Assume $\lambda=-1$.
Then $ a^c=a^{-1} $ and $ b^c=a^jb $ for any integer $j$.
For the case $2\nmid j$ , note that $(bc)^2=bcbc=ba^jb=a^{-j}$ is of order $|a|=2^\ell$, and so $\l bc\r=\ZZ_{2^{\ell+1}}$ is a cyclic subgroup of $ G $ of index $ 2 $, which contradicts our assumption. In fact, $ G=\l a,b,c\r=\l b,c\r=\l bc\r{:}\l c\r=\D_{2^{\ell{+2}}} $. For the case $2\nmid j$, let $c'=a^{j\over2}bc$. Then \[\mbox{$ (c')^2=a^{j\over2}b(a^{j\over2}b)^c=a^{j\over2}ba^{-{j\over2}}a^jb=1 $, and}\] \[\mbox{$ a^{c'}=a^{bc}=a $,\, $ b^{c'}=(a^{-j}b)^{bc}=a^{-j}a^jb=b $.}\] Thus, $G=\l a,b,c\r=\l a,b,c'\r=\l a,b\r\times\l c'\r=\D_{2^{\ell+1}}\times\ZZ_2$, as in part~(2).

(ii.3) Assume $\lambda=2^{\ell-1}+1$.
Then $(2^{\ell-1}+2)j\equiv0$ $(\mod 2^\ell)$. Thus, either $j=0$, \[\mbox{$G=\l a,b\r{:}\l c\r $, with $(a,b)^c=(a^{2^{\ell-1}+1},b)$, as in part~(3);}\] or $j=2^{\ell-1}$, $b^c=a^{2^{\ell-1}}b$. For the latter case, Let $b'=ab$. Then \[(b')^c=(ab)^c=a^{2^{\ell-1}+1}a^{2^{\ell-1}}b=ab=b',\] and so $G=\l a,b,c\r=\l a,b',c\r=\l a,b'\r{:}\l c\r$, with $ (a,b')^c=(a^{2^{\ell-1}+1},b') $, as in part~(3).

(ii.4) Assume $\lambda=2^{\ell-1}-1$.
Then $2^{\ell-1}j\equiv0$ $(\mod 2^\ell)$, so $j$ is even. If $j\equiv0$ $(\mod 4)$, then $ {j\over 2} $ is even, so letting $c'=a^{j\over2}bc$, we have \[ \mbox{$(c')^2=a^{j\over2}b(a^{j\over2}b)^c=(a^{j\over2}b)(a^{{j\over2}(2^{\ell-1}-1)}a^jb)=(a^{j\over2}b)^2=1$, and}\]
\[\mbox{$ a^{c'}=a^{bc}=(a^{-1})^c=a^{2^{\ell-1}+1} $,\, $ b^{c'}=(a^{-j}b)^{bc}=(a^jb)^c=a^{-j}a^jb=b $.}\]
So $G=\l a,b,c\r=\l a,b,c'\r=\l a,b\r{:}\l c'\r$, as in part~(3). Suppose then $j\equiv2$ $(\mod 4)$. Let $b'=ab$, $c'=a^{x}bc$, where $x={j\over2}-2^{\ell-2}$. Then \[\mbox{$(c')^2=a^{x}b(a^{x}b)^c=(a^{x}b)(a^{x(2^{\ell-1}-1)}a^jb)=(a^{x}b)(a^{2^{\ell-1}+{j \over 2}+2^{\ell-2}}b)=(a^{x}b)^2=1$, and}\] \[\mbox{$ a^{c'}=a^{bc}=(a^{-1})^c=a^{2^{\ell-1}+1} $,\, $ b^{c'}=(a^{-2x}b)^{bc}=(a^{2x}b)^{c}=a^{-2x}a^jb=a^{2^{\ell-1}}b$.}\] So, $ (b')^{c'}=a^{c'}b^{c'}=ab=b' $, $G=\l a,b,c\r=\l a,b',c'\r=\l a,b'\r{:}\l c'\r $, as in part~(3).
\end{proof}

\section{Arc-transitive maps}

Let $ G $ be a $ 2 $-group that has a cyclic or dihedral maximal subgroup. In this section, we classify the maps admitting $ G $ as an arc-transitive automorphism group.

Let $ \calM=(V,E,F)$ be a map and $ G\le\Aut\calM $ be a $ 2 $-group acting on $ \calM $. Recall \cite[Lemma~2.2]{HLZZ} (see page 1), by which if $ 4\nmid\chi(\calM) $, then each Sylow $ 2 $-subgroup of $ \Aut\calM $ has a cyclic or dihedral subgroup of index $ 2 $. Note that, $ G $ is contained in some Sylow $ 2 $-subgroup of $ \Aut\calM $. Thus, by \cite[Lemma\,2.5]{HLZZ} -- {\it if a finite group $ G $ satisfies {\rm Hypothesis~$\ref{hypo-0}$}, then so does each subgroup of $ G $} -- we immediately have

\begin{lemma}\label{arc-map}
	Let $ \calM $ be a map and $ G\le\Aut\calM $ be a $ 2 $-group acting on $ \calM $. If $4\nmid\chi(\calM)$, then $ G $ has a cyclic or dihedral subgroup of index $2$, so that, it is one of the groups appearing in {\rm Theorem \ref{thm:p-gps}}.
\end{lemma}

Suppose further that $ \calM $ is \textit{$ G $-arc-transitive}, namely, $ G $ acting transitively on the set of arcs of $ \calM $. Then according to \cite{GW,ST}, the map $ \calM $ lies in one of the five families displayed in the following table, where $ (v,e,f) $ and $ (v',e,f') $ are two {\it flags} (incident triples) with $ v,v'\in V$, $e\in E$, $f,f'\in F $. In other word, one of the following holds:
\[\begin{array}{|c|c|c|c|c|c|}\hline
	\mbox{\footnotesize Type}&G& G_v,\ G_{v'} &G_f,\ G_{f'} &G_e & \mbox{{\footnotesize{remark}}} \\ \hline
	1 & \l x,y,z\r & \l x,y\r & \l y,z\r&\l x,z\r &\mbox{{\footnotesize{regular}}} \\
	2^* & \l x,y,z\r& \l x,y\r &\l x,z\r, \l y,z\r&\l z\r & \mbox{{\footnotesize{vertex-reversing}}}\\
	2^P & \l x,y,z\r&\l x,y\r & \l x,y^z\r &\l z\r & \mbox{{\footnotesize{vertex-reversing}}} \\
	2^*{\rm ex} & \l \a,z\r &\l \a\r &\l z,z^\a\r &\l z\r & \mbox{{\footnotesize{vertex-rotary}}} \\
	2^P{\rm ex} & \l \a,z\r& \l \a\r &\l \a z\r &\l z\r & \mbox{{\footnotesize{vertex-rotary}}} \\ \hline
	%2 &\l a,a',z\r& \l a\r{:}\l z\r,\ \l a'\r{:}\l z\r & \l z^a,z^{a'}\r&\l z\r\\ \hline
\end{array}\] 

(1) $ \calM $ is $ G $-vertex-reversing (type $ 2^\ast $ or $ 2^P $). In this case, $\calM=\calM(G,x,y,z)$ is determined by $ G $, and a \textit{reversing triple} $(x,y,z)$ for $ G $ such that $G=\l x,y,z\r$, $ |x|=|y|=|z|=2 $.

(2) $ \calM $ is $ G $-regular (type 1). In this case, $\calM=\calM(G,x,y,z)$ is determined by $ G $, and a \textit{regular triple} $ (x,y,z) $ for $ G $ such that $ (x,y,z) $ is a reversing triple, $ \l x,z\r\cong\ZZ_2^2 $.

(3) $ \calM $ is $ G $-vertex-rotary (type $ 2^*{\rm ex} $ or $ 2^P{\rm ex} $). In this case, $ \calM=\calM(G,\a,z) $ is determined by $ G $, and a \textit{rotary pair} $ (\a,z) $ for $ G $ such that $G=\l \a,z\r$ and $|z|=2$.\vs

Let $ \Sigma $ be the set of reversing triples for $ G $. Then $ \Aut(G) $ naturally acts on $ \Sigma $ by mapping each $ (x,y,z) $ to $ (x^g,y^g,z^g) $, for any $ g\in\Aut(G) $. Since $ x,y,z $ generate $ G $, this action is \textit{semiregular} with the stabilizer of each triple being the identity. Consequently, $ \Sigma $ is the union of some $ \Aut(G) $-orbits of length $ |\Aut(G)| $. If $ \Sigma $ is taken to be the set of regular triples or rotary pairs, the case is similar. Thus, we have

\begin{lemma}\label{semi}
	Let $ \Sigma $ be the set of reversing triples/regular triples/rotary pairs for $ G $. Then $ \Sigma $ is the union of some $ \Aut(G) $-orbits of length $ |\Aut(G)| $.
\end{lemma}

Moreover, we say that two reversing triples/regular triples/rotary pairs are \textit{equivalent}, if they fall into the same $ \Aut(G) $-orbit. It is easy to see that the maps determined by equivalent triples/pairs are mutually isomorphic. Hence, up to isomorphism, we need only consider the representatives for the $ \Aut(G) $-orbits.

\subsection{Reversing triples \& vertex-reversing maps} 
In this part, let $ G $ be a $ 2 $-group that has a cyclic or dihedral maximal subgroup, and denote respectively by $ \Omega(G) $ and $\calT(G)$ the set of involutions and reversing triples for $ G $. Suppose $ \calT(G)\ne\emptyset $ and let $ (x,y,z)\in\calT(G) $. Note that the involutions $ x,y,z $ are not necessarily distinct.

We thus need only consider those groups from Theorem \ref{thm:p-gps} that can be generated by involutions. In particular, if $ G $ has a cyclic maximal subgroup, then by Proposition \ref{2-gp-cycle}, $G=\ZZ_2^2 $ or $ \D_{2^{\ell+1}} $, $ \ell\ge 2 $. Therefore, $ G $ is one of the following groups:\[\ZZ_2^2,\,\D_{2^{\ell+1}},\,\ZZ_2^3,\,\Q_{2^{\ell+1}}\circ\ZZ_4,\,\D_{2^{\ell+1}}\times\ZZ_2,\,\D_{2^{\ell+1}}{:}\ZZ_2.\]

The next lemma deals with the abelian case, which is easy to prove.

\begin{lemma}\label{revtri1}
	Let $ G=\ZZ_2^2 $ or $\ZZ_2^3$. Then, 
	\begin{itemize}
		\item[(1)] if $ G=\ZZ_2^2=\l a,b\r $, then $ \Delta=\{(a,ab,b),\,(a,b,b),\, (b,a,b),\,(b,b,a)\}$ is a set of representatives for the $ \Aut(G) $-orbits on $ \calT(G) ${\rm;}\vs
		\item[(2)] if $ G=\ZZ_2^3=\l a,b,c\r$, then any reversing triple is equivalent to $(a,b,c)$.
	\end{itemize}
\end{lemma}

\begin{lemma}\label{revtri2}
	Let $G=\D_{2^{\ell+1}}=\l a\r{:}\l b\r$, $ \ell\ge 2 $. Let $ a_0=a^{2^{\ell-1}} $. Then the following set $ \Delta $ is a set of representatives for the $ \Aut(G) $-orbits on $ \calT(G) $, where  \[\Delta=\{(b,ab,w),\,(w,b,ab),\,(ab,w,b)\,|\,w=a_0 \mbox{\ or\ } a^{2x}b,\, 0\le x<2^{\ell-1}\}.\]
\end{lemma}

\begin{proof}
	Let $ (x,y,z)\in\calT(G) $. The set of involutions $ \Omega(G)=\{a_0,a^\lambda b\,|\,0\le \lambda<2^{\ell}\} $. By Proposition \ref{2-gp-cycle}, $ |\Aut(G)|=2^{2\ell-1} $. Note that, $ a_0 $ as the only central involution is fixed by $\Aut(G) $. We thus have two non-equivalent cases as follows:
	
	(i) Suppose $ a_0\in\{x,y,z\} $.  Since $ G=\l x,y,z\r $, the other two involutions must be $ a^ib $ and $ a^jb $ for some $ 0\le i,j<2^{\ell} $ with $ i-j $ odd. Thus, there are precisely $ 3\cdot2^{\ell}\cdot2^{\ell-1}=3\cdot 2^{2\ell-1} $ reversing triples in this case, which by Lemma \ref{semi} fall into $ 3\cdot 2^{2\ell-1}/|\Aut(G)|=3 $ $ \Aut(G) $-orbits, with a set of representatives $ X=\{(b,ab,a_0),\,(a_0,b,ab),\, (ab,a_0,b)\} $.
	
	(ii) Suppose $ a_0\notin\{x,y,z\} $. Then $ (x,y,z)=(a^ib,a^jb,a^kb) $ for some $ 0\le i,j,k<2^{\ell} $. Note that $ (i-j) + (j-k) + (k-i)=0 $. If $ i-j $, $ j-k $, $ k-i $ are all even, then $ i,j,k $ are all even or all odd, so $ \l x,y,z\r$ is contained in the proper subgroup $\l a^2,b\r $ or $ \l a^2,ab\r $ of $ G $, a contradiction. Thus, we have three non-equivalent sub-cases: 
	\begin{itemize}
		\item[(ii.1)] $ i-j $, $ j-k $ are odd, so $ k-i $ is even, $ G=\l x,y\r=\l y,z\r $, $ G\ne\l z,x\r $;
		\item[(ii.2)] $ j-k $, $ k-i $ are odd, so $ i-j $ is even, $ G=\l y,z\r=\l z,x\r $, $ G\ne\l x,y\r $;
		\item[(ii.3)] $ k-i $, $ i-j $ are odd, so $ j-k $ is even, $ G=\l z,x\r=\l x,y\r $, $ G\ne\l y,z\r $.
		\end{itemize}
    For each of the cases, we have $ 2^{\ell}\cdot2^{\ell-1}\cdot2^{\ell-1}=2^{3\ell-2} $ such triples, which by Lemma \ref{semi} fall into $ 2^{3\ell-2}/|\Aut(G)|= 2^{\ell-1} $ $ \Aut(G) $-orbits. Now let \[\mbox{$ Y_1=\{(b,ab,a^{2x}b)\} $, $ Y_2=\{(a^{2x}b,b,ab)\} $, $ Y_3=\{(ab,a^{2x}b,b)\} $, $ 0\le x<2^{\ell-1} $.}\] Note that any $ g\in\Aut(G) $ which fixes $ b $ and also $ ab $ must fix all elements in $ G $. Thus, $Y_1, Y_2, Y_3$ are respectively a set of representatives for the $ 2^{\ell-1} $ orbits of the three cases.
    
    At last, $ \Delta=X\cup Y_1\cup Y_2\cup Y_3$ is what we need.
\end{proof}

\begin{lemma}\label{revtri3}
	Let $G=\Q_{2^{\ell+1}}\circ\ZZ_4=\l a,c\r\circ \l d\r$, defined in {\rm Definition \ref{defi-Q}}. Then 
	\begin{itemize}
		\item[(1)] if $ \ell=2 $, then any reversing triple is equivalent to $ (ad,acd,cd) $;\vs
		\item[(2)] if $ \ell>2 $, then the set $ \Delta= \{(a^{2^{\ell-2}}d,acd,cd),\, (cd,a^{2^{\ell-2}}d,acd),\,(acd, cd,a^{2^{\ell-2}}d)\} $ is a set of representatives for the $ \Aut(G) $-orbits on $ \calT(G) $.
	\end{itemize}
\end{lemma}

\begin{proof}
	Let $ (x,y,z)\in\calT(G) $. By Lemma \ref{ZQ-aut}, $ \Aut(G)\cong\Aut(\Q_{{2^{\ell+1}}})\times\ZZ_2 $. Further, either $ \ell=2 $, $ \Aut(\Q_8)\cong\S_4 $, $ |\Aut(G)|=48 $; or $ \ell>2 $, then by Proposition \ref{2-gp-cycle}, $ \Aut(\Q_{{2^{\ell+1}}})\cong\Hol(\ZZ_{2^{\ell}}) $, $ |\Aut(G)|=2^{2\ell} $ and $ \l a\r $ is fixed by $ \Aut(G) $.
	
	By Lemma \ref{Z-circ-Q-hypo-2}, $\{x,y,z\}$ as a generating triple of involutions of $ G $ is one of $\{a^{2^{\ell-2}}d,a^icd, a^jcd\}$, $ \{a^{2^{\ell-2}}d^{-1},a^icd, a^jcd\} $, where $0\leqslant i,j<2^{\ell}$ and $ i-j $ is odd. Thus, there are precisely $ 2\cdot 3\cdot 2^{\ell}\cdot2^{\ell-1}=3\cdot 2^{2\ell} $ such triples, which by Lemma \ref{semi} fall into $ 3\cdot 2^{2\ell}/||\Aut(G)|=1 $ $\Aut(G)$-orbit if $ \ell=2 $, or into $ 3 $ $\Aut(G)$-orbits if $ \ell>2 $. 
	
	For the former case, choose $ (ad,acd,cd) $ as a representative. For the latter, since $ \l a\r $ is fixed by $ \Aut(G) $ and $ d $ is mapped by $ \Aut(G) $ to $ d^{\pm 1} $, then the element $ a^{2^{\ell-2}}d $ is mapped by $ \Aut(G) $ to some one lying in $ \l a,d^{\pm 1}\r $. Then $ (a^{2^{\ell-2}}d,acd,cd) $, $ (cd,a^{2^{\ell-2}}d,acd) $ and $ (acd, cd,a^{2^{\ell-2}}d) $ form a set of representatives for the 3 orbits.
\end{proof}

\begin{lemma}\label{revtri4}
	Let $ G=\D_{2^{\ell+1}}\times\ZZ_2=\l a, b\r\times\l c\r$, defined in {\rm Definition \ref{defi-direct}}. Then the following set $ \Delta $ is a set of representatives for the $ \Aut(G) $-orbits on $ \calT(G) $, where 
	\[\Delta=\{(b,ab,w),\,(w,b,ab),\,(ab,w,b)\,|\,w=c \mbox{\ or\ } a^{2x}bc,\, 0\le x<2^{\ell-2}\}.\]
\end{lemma}

\begin{proof}
	Let $ (x,y,z)\in\calT(G) $. Let $a_0=a^{2^{\ell-1}}$. By Lemma \ref{Klein-I}, the set of involutions $\Omega(G)=\{a_0,\,c,\,a_0c,\,a^\lambda b,\,a^\lambda bc\,|\,0\le \lambda<2^{\ell}\}$, and the factor group $ \ov G=G/\l a^2\r\cong\ZZ_2^3 $ is not generated by two elements. Then $ \ov G=\l \ov x,\ov y,\ov z\r $ implies that $\ov x,\ov y,\ov z\ne \ov 1$ are pair-wisely distinct with the product of any two not being the remaining one (otherwise $ \l \ov x,\ov y,\ov z\r\cong\ZZ_2^2 $). In particular, $ \ov {a_0}=\ov 1 $, so $ a_0\notin\{x,y,z\} $.
	
	Note that, $ \Z(G)=\Z(\l a,b\r)\times\Z(\l c\r)=\l a_0,c\r\cong \ZZ_2^2 $, which is fixed by $ \Aut(G) $. Then, according to Lemma \ref{2-gp-aut-I}, we have: $ |\Aut(G)|=2^{2\ell+2} $;  $ \Aut(G) $ fixes $ a_0 $ and also the set $ \{c,a_0c\} $. We thus have two non-equivalent cases as follows:
	
	(i) Suppose $\{x,y,z\}\cap\{c, a_0c\}\ne\emptyset$. Since $ \ov c=\ov {a_0c} $, we have $ |\{x,y,z\}\cap\{c,a_0c\}|=1 $. That is, one of $ x,y,z $ is $ c $ or $ a_0c $; meanwhile, the other two involutions $ u,v $ satisfy $ \ov u\ne\ov v $ and $ \ov u\ne \ov {vc} $. It implies that the pair $ (u,v) $ must be one of \[(a^ib,\,a^jb ),\,(a^ibc,\,a^jb ),\,(a^ib,\,a^jbc ),\,(a^ibc,\,a^jbc ),\] where $ 0\le i,j<2^{\ell} $ with $ i-j $ odd. Note that any triple given as above is a reversing triple. Thus, there are precisely $ 2\cdot3\cdot4\cdot2^{\ell}\cdot2^{\ell-1}=3\cdot 2^{2\ell+2} $ reversing triples in this case, which by Lemma \ref{semi} fall into $ 3\cdot 2^{2\ell+2} /|\Aut(G)|=3 $ $ \Aut(G) $-orbits. Recall that $ c $ is mapped by $ \Aut(G) $ to $ c $ or $ a_0c $. Then $ X=\{(b,ab,c),\,(c,b,ab),\, (ab,c,b) $\}  is a set of representatives for the $ 3 $ orbits.
	
	(ii) Suppose $\{x,y,z\}\cap\{c, a_0c\}=\emptyset$. Since $ G=\l x,y,z\r $ and $\ov x,\ov y,\ov z$ are pair-wisely distinct, we have $ \{x,y,z\}=\{a^ib,a^jb,a^kbc\} $ or $ \{a^ibc,a^jbc,a^kb\} $, where $ 0\le i,j,k<2^{\ell} $ and $ i-j $ is odd. Thus, there are precisely $ 2\cdot3\cdot2^{\ell}\cdot2^{\ell-1}\cdot2^{\ell}=3\cdot 2^{3\ell} $ reversing triples in this case, which by Lemma \ref{semi} fall into $ 2^{3\ell}/|\Aut(G)|=3\cdot 2^{\ell-2} $ $ \Aut(G) $-orbits. 
	
	Let $ Y=Y_1\cup Y_2\cup Y_3 $ be a set of reversing triples, where \[\mbox{$ Y_1=\{(b,ab,a^{2x}bc)\} $, $ Y_2=\{(a^{2x}bc,b,ab)\} $, $ Y_3=\{(ab,a^{2x}bc,b)\} $, $ 0\le x<2^{\ell-2}$.}\] Then $ Y $ is a set of representatives for these $ 3\cdot 2^{\ell-2} $ orbits, because 

	(a) (distinct) triples from the same $ Y_i $, w.l.o.g, say $ (b,ab,a^{2y}bc), (b,ab,a^{2y'}bc)\in Y_1 $, are non-equivalent: any $ 1\ne g\in\Aut(G) $ which fixes $ b $ and also $ ab $ must map $ c $ to $ a_0c $, then mapping $ a^{2y}bc $ to $ a^{2y}ba_0c=a^{2y+2^{\ell-1}}bc\ne a^{2y'}bc $, as $ 0\le 2y,2y'<2^{\ell-1} $;
	
    (b) triples from different $ Y_i $ are non-equivalent: if such two triples $ \tt_1,\tt_2 $ are equivalent with some $ g\in\Aut(G) $ such that $ \tt_1^g=\tt_2 $, then by choosing certain coordinates of $ \tt_1,\tt_2 $, we can always have $ (b,ab)^g=(a^{2z}bc,b) $ for some $ 0\le z<2^{\ell-2} $, which is impossible since $ |b\cdot ab|=|a^{-1}|=2^\ell $, but $ |a^{2z}bc\cdot b|=|a^{2z}c|\le 2^{\ell-1} $.
	
	At last, $ \Delta=X\cup Y$ is what we need.
\end{proof}

\begin{lemma}\label{revtri5}
	Let $G=\D_{2^{\ell+1}}{:}\ZZ_2=\l a, b\r{:}\l c\r$, defined in {\rm Definition \ref{defi-semidirect}}.
	Then $ \Delta=X\cup Y $ is a set of representatives for the $ \Aut(G) $-orbits on $ \calT(G) $, where \[X=\{(b,ab,c),\,(ab,b,c),\, (c,b,ab),\,(c,ab,b),\,(ab,c,b),\, (b,c,ab) \},\mbox{\ and}\]
	\[Y=\{(b,ab,w),\,(w,b,ab),\,(ab,w,b)\,|\,w=a^{2x}bc,\, 0\le x<2^{\ell-2}\}.\]
\end{lemma}

\begin{proof}
	The proof follows a similar way as that of Lemma \ref{revtri4}. 
	
	Let $ (x,y,z)\in\calT(G) $, $a_0=a^{2^{\ell-1}}$. By Lemma \ref{Klein-II}, the set of involutions $\Omega(G)=\{a_0,\,c,\,a_0c,\,a^\lambda b,\,a^{2\lambda} bc\,|\,0\le \lambda<2^{\ell}\}$, $ \Z(G)=\l a_0\r\cong \ZZ_2 $, and the factor group $ \ov G=G/\l a^2\r\cong\ZZ_2^3 $. Then $ \ov G=\l \ov x,\ov y,\ov z\r $ implies that $\ov x,\ov y,\ov z\ne \ov 1$ are pair-wisely distinct with the product of any two not being the remaining one. In particular, $ \ov {a_0}=\ov 1 $, so $ a_0\notin\{x,y,z\} $. Note that, $ a_0 $ as the only central element is fixed by $ \Aut(G) $. Then, according to Lemma \ref{2-gp-aut-II}, we have: $ |\Aut(G)|=2^{2\ell} $; $ \Aut(G) $ maps $ c $ to $ c $ or $ a_0c $, so fixes the set $ \{c,a_0c\} $. We thus have two non-equivalent cases as follows:
	
	(i) Suppose $\{x,y,z\}\cap\{c, a_0c\}\ne\emptyset$. Since $ \ov c=\ov {a_0c} $, we have $ |\{x,y,z\}\cap\{c,a_0c\}|=1 $. That is, one of $ x,y,z $ is $ c $ or $ a_0c $; meanwhile, the other two involutions $ u,v $ satisfy $\ov u\ne\ov v$ and $ \ov u\ne \ov {vc} $. It implies that the pair $ (u,v) $ must be one of \[(a^ib,\,a^jb ),\,(a^mb,\,a^{2n}bc),\,(a^{2n}bc,\,a^mb)\] where $ 0\le i,j,m<2^{\ell} $, $ 0\le n<2^{\ell-1} $, and $ i-j,\,m $ are odd. Note that any triple given as above is a reversing triple. Thus, there are precisely $ 2\cdot3\cdot(2^{\ell}\cdot2^{\ell-1}+2\cdot2^{\ell-1}\cdot2^{\ell-1})=3\cdot 2^{2\ell+1} $ reversing triples in this case, which by Lemma \ref{semi} fall into $ 3\cdot 2^{2\ell+1} /|\Aut(G)|=6 $ $ \Aut(G) $-orbits. Let $ X=X_1\cup X_2 \cup X_3 $, where  \[X_1=\{(b,ab,c),\,(ab,b,c)\},\,X_2= \{(c,b,ab),\,(c,ab,b)\}, \, X_3=\{(ab,c,b),\, (b,c,ab)\}.\] Since $ c $ is mapped by $ \Aut(G) $ to $ c $ or $ a_0c $, triples lying in different $ X_i $ are mutually non-equivalent. Meanwhile, $ [b,c]=1 $ but $ [ab,c]=a_0\ne 1 $, so the two triples lying in the same $ X_i $ are non-equivalent. Then $ X $ is a set of representatives for the $ 6 $ orbits.
	
	(ii) Suppose $\{x,y,z\}\cap\{c, a_0c\}=\emptyset$. Since $ G=\l x,y,z\r $ and $\ov x,\ov y,\ov z$ are pair-wisely distinct, we have $ \{x,y,z\}=\{a^ib,a^jb,a^{2k}bc\} $, where $ 0\le i,j<2^{\ell} $, $ 0\le k<2^{\ell-1} $ and $ i-j $ is odd. Thus, there are precisely $ 3\cdot2^{\ell}\cdot2^{\ell-1}\cdot2^{\ell-1}=3\cdot 2^{3\ell-2} $ reversing triples in this case, which by Lemma \ref{semi} fall into $ 3\cdot 2^{3\ell-2} /|\Aut(G)|=3\cdot 2^{\ell-2} $ $ \Aut(G) $-orbits. It is similar as Lemma \ref{revtri4} to show: $ Y=Y_1\cup Y_2\cup Y_3 $ with \[\mbox{$ Y_1=\{(b,ab,a^{2x}bc)\} $, $ Y_2=\{(a^{2x}bc,b,ab)\} $, $ Y_3=\{(ab,a^{2x}bc,b)\} $, $ 0\le x<2^{\ell-2} $},\] is a set of representatives for these orbits.
	
	At last, $ \Delta=X\cup Y$ is what we need.
\end{proof}

\begin{proposition}\label{2gp-rev-maps}
	Let $ G $ be a finite $ 2 $-group and $ \calM $ be a $ G $-vertex-reversing map with $4\nmid\chi(\calM)$. Then $ G $ and $ \calM=\calM(G,x,y,z) $ are listed in the following table, where $ [x,y,z] $ denotes  the set $ \{(x,y,z),(z,x,y),(y,z,x)\} $.
	\begin{table}[h]
		\newcommand{\tabincell}[2]{\begin{tabular}{@{}#1@{}}#2\end{tabular}}
		
		\caption{\small $ G $-vertex-reversing maps}
		\centering
		\scalebox{0.8}{
			\begin{tabular}{ccccc}
				\toprule[1pt]
				$G$ & $(x,y,z)$ &  {\rm type}  & $ \chi(\calM)$ & {\rm remark}\\
				\midrule[1pt]
				
				$ \ZZ_2^2=\l a,b\r $ & \tabincell{c}{$ (a,ab,b) $ \\ $[a,b,b]$ \\ $(a,a,b)$} &
				\tabincell{c}{$ 2^\ast $ \\ $ 2^\ast $ \\ $ 2^P $} & \tabincell{c}{$ 1 $ \\ $ 2 $ \\ $ 2$} & \\
				\midrule[0.5pt]
				
				$ \ZZ_2^3=\l a,b, c\r $ & $ (a,b,c) $ & $ 2^\ast $ 
				& $ 2 $ &  \\
				\midrule[0.5pt]
				
				\tabincell{c}{$\D_{2^{\ell+1}}=\l a\r{:}\l b\r$} & \tabincell{c}{$[b,ab,a_0] $ \\ $[b,ab,a^{2x}b] $ \\ $ (b,ab,a_0),(b,ab,a^{2x}b),(ab,a^{2x}b,b) $ \\ $ (a^{2x}b,b,ab) $} & 
				\tabincell{c}{$ 2^\ast $ \\$ 2^\ast $ \\ $ 2^P $ \\ $ 2^P $} &
				\tabincell{c}{$1$ \\ $ 2-2^{\ell}+2^{s} $ \\ $2-2^{\ell}$ \\ $ 2-2^{\ell}+2^{s} $} & \tabincell{c}{ \\ $ 2\le s\le \ell $ \\ \\ $ 2\le s\le \ell $ } \\
				\midrule[0.5pt]
			    
			    $ \Q_{8}\circ\ZZ_4=\l a,c\r\circ\l d\r $ & $ (ad,acd,cd) $ & $ 2^\ast $ 
			    & $ -2 $ & \\
			    \midrule[0.5pt]
			     
			    $ \Q_{2^{\ell+1}}\circ\ZZ_4=\l a,c\r\circ \l d\r
			     $ & $ 
			     [a^{2^{\ell-2}}d,cd,acd] $ & $ 2^\ast $ 
			    & $ 2-2^{\ell} $ & $ \ell\ge3 $\\
			    \midrule[0.5pt]
			    
			    $ \D_{2^{\ell+1}}\times\ZZ_2=\l a, b\r\times\l c\r $ & $ 
			    [b,ab,c] $ & $ 2^\ast $ 
			    & $ 2 $ & \\
			    \midrule[0.5pt]
			    
			    $ \D_{2^{\ell+1}}{:}\ZZ_2=\l a, b\r{:}\l c\r $ & $ 
			    [b,ab,c],\, [ab,b,c] $ & $ 2^\ast $ 
			    & $ 2-2^{\ell-1} $ & \\
				\bottomrule[1pt]
		\end{tabular}}
	\end{table}
\end{proposition}

\begin{proof}
	By Lemma \ref{arc-map}, $ G $ is one of the groups appearing in Theorem \ref{thm:p-gps}. Now, the reversing triples for such groups are determined in Lemma \ref{revtri1}\,-\,\ref{revtri5}. We therefore need only select, from all corresponding maps, those satisfying the condition $ 4\nmid\chi(\calM) $. Recall that, respectively for type $ 2^* $ and $ 2^P $ we have \[\mbox{$\chi_1(\calM)={|G|\over|\l x,y\r|}-{|G|\over2}+{|G|\over|\l x,z\r|}+{|G|\over|\l y,z\r|}$, and $\chi_2(\calM)={|G|\over| \l x,y\r|}-{|G|\over2}+{|G|\over| \l x,y^z\r|}$}.\] 
	
	For convenience, denote simply by $ \tt$ a reversing triple. For a given triple $(x,y,z) $, respectively denote by $ a_1, a_2, a_3 $, $ a_4 $ the number $ |\l x,y\r|,|\l y,z\r|, |\l z,x\r| $, $ |\l x,y^z\r| $, and by $ [x,y,z] $ the set of triples $ \{(x,y,z),(z,x,y),(y,z,x)\} $ obtained by cyclic permutations of $ (x,y,z) $. Note that, $ \chi_1(\calM) $ takes the same value for the three triples lying in $ [x,y,z]$.
	
	The small-order case $ G=\ZZ_2^2 $ or $ \ZZ_2^3 $ is simple, directly presented in the table.
	
	(i) Suppose $G=\D_{2^{\ell+1}}=\l a\r{:}\l b\r$, $ \ell\ge 2 $. Let $ a_0=a^{2^{\ell-1}} $. 
	
	(i.1) Let $\tt\in[b,ab,a_0]$. Then $ \{a_1,a_2,a_3\}=\{4,4,2^{\ell+1}\} $, so $ \chi_1(\calM)=1 $. Meanwhile, if $ \tt=(b,ab,a_0) $ with $ (ab)^{a_0}=ab $, then $ a_1=a_4=2^{\ell+1} $ ; if $ \tt=(a_0,b,ab) $, $ (ab,a_0,b) $ with $ b^{ab}=a^2b $, $ a_0^b=a_0 $, then $ a_1=a_4=4 $, so respectively $ \chi_2(\calM)=2-2^{\ell} $ or $ 0 $.
	
	(i.2) Let $\tt\in[b,ab,a^{2x}b]$, $ 0\le x<2^{\ell-1} $. Then $ \{a_1,a_2,a_3\}=\{2^{\ell+1},2^{\ell+1},2|a^{2x}|\} $, where $ |a^{2x}|=2^{\ell-s} $, with $ s $ the largest number within $ 1\le s\le\ell $ such that $ 2^{s-1}\mid x $, so $ \chi_1(\calM)=2-2^{\ell}+2^s $. Note that, if $ s>1 $, i.e., $ 2\mid x $, then $ 4\nmid \chi_1(\calM) $. Meanwhile, if $ \tt=(b,ab,a^{2x}b) $, $ (ab,a^{2x}b,b) $ with $ (ab)^{a^{2x}b}=a^{4x-1}b $, $ (a^{2x}b)^b=a^{-2x}b $, then $ a_1=a_4=2^{\ell} $, so $ \chi_2(\calM)=2-2^\ell $; if $ \tt=(a^{2x}b,b,ab) $ with $ b^{ab}=a^2b $, then $ (a_1,a_4)=(2|a^{2x}|,2|a^{2x-2}|) $. Note that, one and only one of $ 2x, 2x-2 $ is not divisible by $ 4 $, so one of $ a_1,a_4 $ is equal to $ 2^{\ell} $, while the other is equal to $ 2^t $ with $ 1\le t\le \ell-1 $, so $ \chi_2(\calM)=2-2^\ell+2^s $, where $ 2\le s\le \ell $.\vs
	
	(ii) Suppose $G=\Q_{2^{\ell+1}}\circ\ZZ_4=\l a,c\r\circ \l d\r$. The degenerate case $ \ell=2 $ fits the general $\ell>2$ treatment. Let $ \tt\in [a^{2^{\ell-2}}d,acd,cd] $. Since $ a^{2^{\ell-2}}d\cdot acd=a^{2^{\ell-2}+1}cd^2$, $ a^{2^{\ell-2}}d\cdot cd=a^{2^{\ell-2}}cd^2$ are of order $ 4 $, and $acd\cdot cd=a $ is of order $ 2^{\ell} $, so $ \{a_1,a_2,a_3\}=\{8,8,2^{\ell+1}\} $, and then $ \chi_1(\calM)=2-2^{\ell} $. Meanwhile, if $ \tt=(a^{2^{\ell-2}}d,acd,cd), (cd,a^{2^{\ell-2}}d,acd) $ with $ (acd)^{cd}=a^{-1}cd $, $ (a^{2^{\ell-2}}d)^{acd}=a^{-2^{\ell-2}}d $, then $ a_1=a_4=8 $, so $ \chi_2(\calM)=-2^{\ell} $; if $ \tt=(acd,cd,a^{2^{\ell-2}}d) $ with $ (cd)^{a^{2^{\ell-2}}d}=c^3d $, then $ a_1=a_4=2^{\ell+1} $, so $ \chi_2(\calM)=4-2^{\ell+1} $.\vs

	(iii) Suppose $G=\D_{2^{\ell+1}}\times\ZZ_2=(\l a\r{:}\l b\r)\times\l c\r$, $ \ell\ge 2 $. 
	
	(iii.1) Let $\tt\in[b,ab,c]$. Then $ \{a_1,a_2,a_3\}=\{4,4,2^{\ell+1}\} $, so $ \chi_1(\calM)=2 $. Meanwhile, if $ \tt=(b,ab,c) $ with $ (ab)^{c}=ab $, then $ a_1=a_4=2^{\ell+1} $ ; if $ \tt=(c,b,ab) $, $ (ab,c,b) $ with $ b^{ab}=a^2b $, $ c^b=c $, then $ a_1=a_4=4 $, so respectively $ \chi_2(\calM)=4-2^{\ell+1} $ or $ 0 $.
	
	(iii.2) Let $\tt\in[b,ab,a^{2x}bc]$, $ 0\le x<2^{\ell-2} $. Then $ \{a_1,a_2,a_3\}=\{2^{\ell+1},2^{\ell+1},2|a^{2x}c|\} $, so $ \chi_1(\calM)=4-2^{\ell+1}+{2^{\ell+2}\over2|a^{2x}c|} $ is divisible by $ 4 $, as $ |a^{2x}c|\le 2^{\ell-1} $. Meanwhile, if $ \tt=(b,ab,a^{2x}bc) $, $ (ab,a^{2x}bc,b) $ with $ (ab)^{a^{2x}bc}=a^{4x-1}b $, $ (a^{2x}bc)^b=a^{-2x}bc $, then $ a_1=a_4=2^{\ell+1} $, so $ \chi_2(\calM)=4-2^{\ell+1} $; if $ \tt=(a^{2x}bc,b,ab) $ with $ b^{ab}=a^2b $, then $ (a_1,a_4)=(2|a^{2x}c|,2|a^{2x-2}c|) $, so $ \chi_2(\calM)={2^{\ell+2}\over2|a^{2x}c|}-2^{\ell+1}+{2^{\ell+2}\over2|a^{2x-2}c|} $ is divisible by $ 4 $, as $ |a^{2x}c|,|a^{2x-2}c|\le 2^{\ell-1} $.\vs
	
	(iv) Suppose $G=\D_{2^{\ell+1}}{:}\ZZ_2=(\l a\r{:}\l b\r){:}\l c\r$ with $(a,b)^c=(a_0a,b)$, $ a_0=a^{2^{\ell-1}} $, $ \ell\ge 3 $. 
	
	(iv.1) Let $\tt\in[b,ab,c]$ or $ [ab,b,c]$. As $ (abc)^2=ab(ab)^c=a_0 $, $ |abc|=4 $. Then $ \{a_1,a_2,a_3\}=\{4,8,2^{\ell+1}\} $, so $ \chi_1(\calM)=2-2^{\ell-1} $. Meanwhile, if $ \tt=(b,ab,c),(ab,b,c) $ with $ (ab)^{c}=a_0ab, b^c=b $, then $ a_1=a_4=2^{\ell+1} $ ; if $ \tt=(c,b,ab),(ab,c,b), (b,c,ab) $ with $ (b)^{ab}=a^2b, c^b=b, c^{ab}=a_0c$, then $ a_1=a_4=4 $ ; if $ \tt=(c,ab,b) $ with $ (ab)^{b}=a^{-1}b $, then $ |abc|=|a^{-1}bc|=4$, $a_1=a_4=8 $, so respectively $ \chi_2(\calM)=4-2^{\ell+1} $, $ 0 $ or $-2^{\ell}$.
	
	(iv.2) Let $\tt\in[b,ab,a^{2x}bc]$, $ 0\le x<2^{\ell-2} $. Then $ \{a_1,a_2,a_3\}=\{2^{\ell+1},2^{\ell+1},2|a^{2x}c|\} $, so $ \chi_1(\calM)=4-2^{\ell+1}+{2^{\ell+2}\over2|a^{2x}c|} $ is divisible by $ 4 $, as $ [a^2,c]=1 $ and so $ |a^{2x}c|\le 2^{\ell-1} $. Meanwhile, if $ \tt=(b,ab,a^{2x}bc) $, $ (ab,a^{2x}bc,b) $ with $ (ab)^{a^{2x}bc}=a_0a^{4x-1}b $, $ (a^{2x}bc)^b=a^{-2x}bc $, then $ a_1=a_4=2^{\ell+1} $, so $ \chi_2(\calM)=4-2^{\ell+1} $; if $ \tt=(a^{2x}bc,b,ab) $ with $ b^{ab}=a^2b $, then $ (a_1,a_4)=(2|a^{2x}c|,2|a^{2x-2}c|) $, so $ \chi_2(\calM)={2^{\ell+2}\over2|a^{2x}c|}-2^{\ell+1}+{2^{\ell+2}\over2|a^{2x-2}c|} $ is divisible by $ 4 $, as $ [a^2,c]=1 $ and so $ |a^{2x}c|,|a^{2x-2}c|\le 2^{\ell-1} $.
\end{proof}

\subsection{Regular triples \& regular maps} In Lemma \ref{revtri1}\,-\,\ref{revtri5}, for $ G $ a $ 2 $-group having a cyclic or dihedral maximal subgroup, we have determined the set $ \calT(G) $ of reversing triples for $ G $, with a set of representatives $ \Delta $ given for the $ \Aut(G) $-orbits on $ \calT(G) $. Denote by $ \calT_0(G) $ the set of regular triples for $ G $. By definition, $ \calT_0(G)\subset\calT(G) $, and a set of representatives $ \Delta_0 $ for the $ \Aut(G) $-orbits on $ \calT_0(G) $ is simply obtained by choosing the triples $ (x,y,z) $ from $ \Delta $ satisfying: $ \l x,z\r\cong\ZZ_2^2 $, i.e., $ [x,z]=1 $ and $ x\ne z $. Consequently, we have

\begin{corollary}\label{regtri}
	Let $ G $ be a $ 2 $-group that has a cyclic or dihedral maximal subgroup. Suppose $\calT_0(G)\ne\emptyset$, that is, $ G $ having a regular triple. Then $ G $ is one of the following groups, with $ \Delta_0 $ a set of representatives for the $ \Aut(G) $-orbits on $ \calT_0(G) $ given.
	\begin{itemize}
		\item[\rm (1)] $ G=\ZZ_2^2=\l a,b\r $, $ \Delta_0=\{(a,ab,b),\,(a,b,b),\,(b,b,a)\} ${\rm ;}\vs
		\item[\rm (2)] $ G=\ZZ_2^3=\l a,b,c\r $, $ \Delta_0=\{(a,b,c)\}${\rm ;}\vs
		\item[\rm(3)] $G=\D_{2^{\ell+1}}=\l a\r{:}\l b\r$, $ \Delta_0=\{(b,ab,a_0),\,(b,ab,a_0b),\, (a_0,b,ab)\}$, $a_0=a^{2^{\ell-1}}${\rm ;}\vs
		\item[\rm (4)] $ G=\D_{2^{\ell+1}}\times\ZZ_2=(\l a\r{:}\l b\r)\times\l c\r$, $ \Delta_0=\{(b,ab,c),\,(b,ab,bc),\, (c,b,ab)\}$.\vs
		\item[\rm (5)] $ G=\D_{2^{\ell+1}}{:}\ZZ_2=(\l a\r{:}\l b\r){:}\l c\r$ with $ (a,b)^c=(a^{2^{\ell-1}+1},b) $, then the set $ \Delta_0 $ is $ \Delta_0=\{(b,ab,c),\,(c,ab,b),\, (b,ab,bc)\}$.
	\end{itemize}
\end{corollary}

\noindent{\it Proof.} 
    For abelian case $ G=\ZZ_2^2 $ or $ \ZZ_2^3 $, it is trivial by Lemma \ref{revtri1}; for $ G= \Q_{2^{\ell+1}}\circ\ZZ_4 $,  $ G $ has no regular triples by Lemma \ref{Z-circ-Q-hypo-2} (or Lemma \ref{revtri3}).
	
	For $ G=\D_{2^{\ell+1}}=\l a\r{:}\l b\r $, by Lemma \ref{revtri2} we have \[ \Delta=\{(b,ab,w),(w,b,ab),(ab,w,b)\,|\,w=a_0 \mbox{\ or\ } a^{2x}b, 0\le x<2^{\ell-1}\},\] and then $ \Delta_0=\{(b,ab,a_0),\,(b,ab,a_0b),\, (a_0,b,ab)\} $, because 
	\begin{itemize}
		\item[(i)] $ [b,a_0]=1 $, and $ [b,a^{2x}b]=a^{-4x}=1 $ implies $ x=0$ or $ 2^{\ell-2} $, $ a^{2x}b=b$ or $ a_0b $;
		\item[(ii)] $ [a_0,ab]=1 $, $ [a^{2x}b,ab]=a^{4x-2}\ne 1 $; and $ [ab,b]=a^2\ne 1 $.
	\end{itemize}  
	
	For $ G=\D_{2^{\ell+1}}\times\ZZ_2=(\l a\r{:}\l b\r)\times\l c\r$, by Lemma \ref{revtri4} we have \[ \Delta=\{(b,ab,w),\,(w,b,ab),\,(ab,w,b)\,|\,w=c \mbox{\ or\ } a^{2x}bc,\, 0\le x<2^{\ell-2}\},\] and then $ \Delta_0=\{(b,ab,c),\,(b,ab,bc),\, (c,b,ab)\}$, because 
	\begin{itemize}
		\item[(i)] $ [b,c]=1 $, and $ [b,a^{2x}bc]=(a^{-2x}c)^2=a^{-4x}=1 $ implies $ x=0 $, $ a^{2x}bc=bc $;
		\item[(ii)] $ [c,ab]=1 $, $ [a^{2x}bc,ab]=(a^{2x-1}c)^2=a^{4x-2}\ne 1 $; and $ [ab,b]=a^2\ne 1 $.
	\end{itemize}
	
	For $ G=\D_{2^{\ell+1}}{:}\ZZ_2=(\l a\r{:}\l b\r){:}\l c\r$ with $ (a,b)^c=(a_0a,b) $, $a_0=a^{2^{\ell-1}} $, by Lemma \ref{revtri5} we have $ \Delta=X\cup Y $, where \[X=\{(b,ab,c),\,(ab,b,c),\, (c,b,ab),\,(c,ab,b),\,(ab,c,b),\, (b,c,ab) \},\mbox{\ and}\]
	\[Y=\{(b,ab,w),\,(w,b,ab),\,(ab,w,b)\,|\,w=a^{2x}bc,\, 0\le x<2^{\ell-2}\},\]
	and then $ \Delta_0=\{(b,ab,c),\,(c,ab,b),\, (b,ab,bc)\}$, because
	\begin{itemize}
		\item[(i)] $ [b,c]=1$, $[ab,c]=ab(ab)^c=a_0\ne1$, $[ab,b]=a^2\ne 1$;
		\item[(ii)] $ [b,a^{2x}bc]=(a^{-2x}c)^2=a^{-4x}=1 $ implies $ x=0 $, $ a^{2x}bc=bc $;
		\item[(iii)] $ [a^{2x}bc,ab]=a^{2x}b(aba^{2x}b)^cab=a^{2x}b(a_0a^{1-2x})ab=a_0a^{4x-2}\ne 1 $.\qed
	\end{itemize}

\begin{proposition}\label{2gp-reg-maps}
	Let $ G $ be a finite $ 2 $-group and $ \calM $ be a $ G $-regular map with $4\nmid\chi(\calM)$. Then $ G $ and $ \calM=\calM(G,x,y,z) $ are listed in the following table.
	\begin{table}[h]
		\newcommand{\tabincell}[2]{\begin{tabular}{@{}#1@{}}#2\end{tabular}}
		\centering
	    \caption{\small $ G $-regular maps}
			\scalebox{0.8}{
				\begin{tabular}{cccc}
				\toprule[1pt]
				$ \calM $ & $G$ & $(x,y,z)$ & $ \chi(\calM)$\\
				\midrule[1pt]
				
				\tabincell{c}{$ \EM_1(2) $ \\ $ \EM_5(2), \EM_6(2) $ } & $ \ZZ_2^2=\l a,b\r $ & \tabincell{c}{$ (a,ab,b) $ \\ $(a,b,b)$ or $(b,b,a)$}  & \tabincell{c}{$ 1 $\\ $ 2 $}\\
				\midrule[0.5pt]
				
				\tabincell{c}{$ \EM_1(2^{\ell}), \EM_2(2^{\ell}) $ \\ $ \EM_4(2^{\ell}) $, $ \ell\ge3 $} & $\D_{2^{\ell+1}}=\l a\r{:}\l b\r$ & \tabincell{c}{$(b,ab,a_0) $ or $(a_0,b,ab)$ \\ $(b,ab,a_0b)$}  & \tabincell{c}{$1$\\$ 2-2^{\ell-1} $}\\
				\midrule[0.5pt]
				
				$ \M_{2,2}(2) $ & $ \ZZ_2^3=\l a,b, c\r $ & $ (a,b,c) $ & $ 2 $ \\
				\midrule[0.5pt]
				
				$2^\ell$-dipole or $2^\ell$-cycle & $\D_{2^{\ell+1}}{\times}\ZZ_2=\l a,b\r{\times}\l c\r$ & $ (b,ab,c) $ or $(c,b,ab)$ & $2$\\
				\midrule[0.5pt]
				
				&$\D_{2^{\ell+1}}{:}\ZZ_2=\l a,b\r{:}\l c\r$ & $ (b,ab,c) $ or $(c,ab,b)$ & $ 2-2^{\ell-1} $\\
				\bottomrule[1pt]
		\end{tabular}}
	\end{table}
\end{proposition}

\begin{proof}
    By Lemma \ref{arc-map}, $ G $ is one of the groups appearing in Theorem \ref{thm:p-gps}. Now, the regular triples for such groups are determined in Corollary \ref{regtri}. We therefore need only select, from all corresponding maps, those satisfying the condition $ 4\nmid\chi(\calM) $. Recall that, for type 1 we have \[\mbox{$\chi(\calM)={|G|\over|\l x,y\r|}-{|G|\over4}+{|G|\over|\l y,z\r|}$}.\]
	For convenience, denote simply by $ \tt$ a regular triple. The small-order case $ G=\ZZ_2^2 $ or $ \ZZ_2^3 $ is simple, directly presented in the table.
	
	(i) Suppose $G=\D_{2^{\ell+1}}=\l a\r{:}\l b\r$, $ \ell\ge 2 $. Let $ a_0=a^{2^{\ell-1}} $. Since $ \l b,ab\r=\l a_0b,ab\r=G $ and $ \l a_0,ab\r\cong\ZZ_2^2 $, then either $ \tt=(b,ab,a_0) $ or $(a_0,b,ab)$, $ \chi(\calM)=1 $; or $ \tt=(b,ab,a_0b) $, $ \chi(\calM)=2-2^{\ell} $, which is not divisible by $ 4 $ as $ \ell\ge 2 $.
	
	(ii) Suppose $G=\D_{2^{\ell+1}}\times\ZZ_2=(\l a\r{:}\l b\r)\times\l c\r$, $ \ell\ge 2 $. Since $ |ab\cdot bc|=|ac|=|a| $, then $ \l ab,bc\r\cong\D_{2^{\ell{+1}}} $. Meanwhile, $ \l ab,c\r\cong\ZZ_2^2 $. Thus, either $ \tt=(b,ab,c) $ or $ (c,b,ab) $, $ \chi(\calM)=2 $; or $ \tt=(b,ab,bc) $, $ \chi(\calM)=4-2^{\ell} $, which is divisible by $ 4 $ as $ \ell\ge 2 $.
	
	(iii) Suppose $G=\D_{2^{\ell+1}}{:}\ZZ_2=(\l a\r{:}\l b\r){:}\l c\r$ with $(a,b)^c=(a_0a,b)$, $ a_0=a^{2^{\ell-1}} $, $ \ell\ge 3 $. Since $ (ac)^2=aa^c=a_0a^2$ is of order $ 2^{\ell-1} $, then $ |ab\cdot bc|=|ac|=2^{\ell} $, $ \l ab,bc\r\cong\D_{2^{\ell{+1}}} $. Meanwhile, since $(abc)^2=ab(ab)^c=aba_0ab=a_0$ is of order $ 2 $, then $ |abc| = 4 $, $ \l ab,c\r\cong\D_{8} $. Thus, either $ \tt=(b,ab,c) $ or $ (c,ab,b) $, $ \chi(\calM)=2-2^{\ell-1} $; or $ \tt=(b,ab,bc) $, $ \chi(\calM)=4-2^{\ell} $. As $ \ell\ge 3 $, we have $ 4\,\nmid\,(2-2^{\ell-1}) $  and $ 4\,\mid\,(4-2^{\ell}) $.
\end{proof}

\noindent{\bf Remark.} The marks in column 1 of the table are collected from \cite{unfaithful}. If $ \calM $ lies in rows 1\,-\,6, then the edge-stabilizer $ G_e=\l x,z\r $ is not core-free in $ G $, and so $ G $ acts unfaithfully on edges of $ \calM $. Refer to \cite{unfaithful} for an investigation of such maps. In particular, if $ \calM $ lies in rows 1\,-\,4, where $G=\ZZ_2^2$ or $ \D_{2^{\ell+1}} $ is generated by less than three involutions, then $ \calM $ is called a \textit{redundant} regular map, with a mark $ \EM $.

\subsection{Rotary pairs \& vertex-rotary maps} As before, let $ G $ be a $ 2 $-group that has a cyclic or dihedral maximal subgroup, and denote by $ \Omega(G) $ the set of involutions of $ G $. Suppose $ G $ having a rotary pair, say $ (\a,z) $. Then $G=\l \a,z\r$, where $ z \in\Omega(G)$.

We thus need only consider those groups from Theorem \ref{thm:p-gps} that can be generated by two elements. Respectively by Lemma \ref{Z-circ-Q-hypo-2}, \ref{Klein-I} and \ref{Klein-II} we have $ G \ne \Q_{2^{\ell+1}}\circ\ZZ_4 $, 
$ \D_{2^{\ell+1}}\times\ZZ_2 $ or $ \D_{2^{\ell+1}}{:}\ZZ_2 $. Further, $ G\ne\Q_{{2^{\ell+1}}} $, as the only involution is the central involution. Therefore, $ G $ is one of the groups: \[\ZZ_{2^{\ell}},\,\ZZ_{2^{\ell}}\times\ZZ_2,\,\D_{{2^{\ell+1}}},\,\SD_{{2^{\ell+1}}},\,\ZZ_{2^{\ell}}{:}\ZZ_2,\] where the non-cyclic groups can be uniformly expressed in the following form: \[\mbox{$ G=\l a\r{:}\l b\r=\ZZ_{2^\ell}{:}\ZZ_2$,\, $a^b=a^\lambda$, $\lambda\in\{\pm 1, 2^{\ell-1}\pm 1\}$,}\] with $ \ell\ge1 $ if $ \lambda=1 $, $ \ell\ge2 $ if $ \lambda=-1 $, and $ \ell\ge3 $ if $ \lambda=2^{\ell-1}\pm 1 $.

The next lemma is easy to prove.

\begin{lemma}\label{rotpair1} If $ G=\ZZ_{2^{\ell}}=\l a\r $, $ \ell\ge 2 $, then $ (\a,z) $ is equivalent to $(a,a^{2^{\ell-1}})$.
\end{lemma}

\begin{lemma}\label{rotpair2}
	Let $ G=\l a\r{:}\l b\r=\ZZ_{2^\ell}{:}\ZZ_2$, with $a^b=a^\lambda$, $\lambda\in\{\pm 1,\,2^{\ell-1}\pm 1 \}$. Then
	\begin{itemize}
		\item[\rm (1)] if $ \lambda\in\{1,2^{\ell-1}+1\}$, then $ (\a,z) $ is equivalent to $(a,b)${\rm;}\vs 
		\item[\rm (2)] if $ \lambda\in\{-1,2^{\ell-1}-1\}$, then $ (\a,z) $ is equivalent to $(a,b)$ or $ (ab,b) $.
	\end{itemize} 
\end{lemma}

\begin{proof}
	First, if $ \ell=1 $, $ G=\ZZ_2^2 $, then obviously $ (\a,z) $ is equivalent to $ (a,b) $. We then assume $ \ell\ge 2 $, and let $ a_0=a^{2^{\ell-1}} $. According to Proposition \ref{2-gp-cycle}, with the small-order case noted as $ \Aut(\ZZ_4\times\ZZ_2)=\Aut(\D_8)=\D_8 $,  we conclude $ |\Aut(G)|= 2^{\ell+1}$, $2^{2\ell-1}$ or $2^{2\ell-2}$, respectively for $\lambda\in\{1, 2^{\ell-1}+1\}$, $\lambda=-1$ or $\lambda=2^{\ell-1}-1$.
	
	Note that $ (a^2)^b=(a^2)^\lambda $, so $ \l a^2\r\lhd G $. Let $ \ov G=G/\l a^2\r=\l \ov a,\ov b\r\cong\ZZ_2^2 $, so that, $ \ov G=\l \ov {a},\ov {z}\r $ implies $\ov\a,\ov z\in\{\ov a,\ov b,\ov {ab}\}$ and $ \ov \a\ne\ov z $.
	
	(i) Suppose $ \lambda=1$ or $2^{\ell-1}+1$. By simple calculation, the set of involutions $ \Omega(G)=\{a_0,a_0b,b\} $. Thus, $ z\in\{a_0b,\, b\} $, $ \ov z=\ov b $, and so $ \ov\a=\ov a $ or $ \ov {ab} $, namely, $ \a=a^i $ or $ a^ib $ with $ i $ odd. Each of such pairs can generate $ G $. Thus, there are precisely $ 2\cdot 2\cdot 2^{\ell-1}=2^{\ell+1} $ rotary pairs, which by Lemma \ref{semi} fall into $ 2^{\ell+1}/||\Aut(G)|=1 $ orbit, with $ (a,b) $ a representative.
	
	(ii) Suppose $ \lambda=-1$. Then $ G $ is dihedral, $ \Omega(G)=\{a_0,a^{\lambda}b\,|\,0\le \lambda<2^{\ell}\}$. Thus, $ z=a^{j}b$ for some $0\le j<2^{\ell}$, and since $ \ov\a\ne\ov 1,\ov z $, either $ \ov \a=\ov a $, $ \a=a^i $ with $ i $ odd, or $\a= a^{k}b $ with $k-j $ odd. Each of such pairs can generate $ G $. Thus, there are precisely $ 2^{\ell}\cdot 2^{\ell-1}+2^{\ell}\cdot 2^{\ell-1}=2^{2\ell} $ rotary pairs, which by Lemma \ref{semi} fall into $ 2^{2\ell}/||\Aut(G)|=2 $ orbits, with representatives $(a,b)$ and $(ab,b)$.
	
	(iii) Suppose $ \lambda=2^{\ell-1}-1$. Then $ G $ is semi-dihedral, $ \Omega(G)=\{a_0,a^{2\lambda}b\,|\,0\le \lambda<2^{\ell-1}\}$. Thus, $ z=a^{2k}b$ for some $0\le k<2^{\ell-1} $, $\ov z=\ov b$, and so $ \ov\a=\ov a $ or $ \ov {ab} $, namely, $ \a=a^i $ or $ a^ib $ with $ i $ odd. Each of such pairs can generate $ G $. Thus, there are precisely $ 2^{\ell-1}\cdot2\cdot 2^{\ell-1}=2^{2\ell-1} $ rotary pairs, which by Lemma \ref{semi} fall into $ 2^{2\ell-1}/||\Aut(G)|=2 $ orbits, with representatives $(a,b)$ and $(ab,b)$.
\end{proof}

\begin{proposition}\label{2gp-rot-maps}
	Let $ G $ be a finite $ 2 $-group and $ \calM $ be a $ G $-vertex-rotary map with $ 4\nmid\chi(\calM) $. Then $ G $ and $ \calM=\calM(G,\a,z) $ are listed in the following table.
	\begin{table}[h]
		\newcommand{\tabincell}[2]{\begin{tabular}{@{}#1@{}}#2\end{tabular}}
		\centering
		\caption{\small $ G $-vertex-rotary maps}
		\renewcommand{\arraystretch}{0.9}
		\scalebox{0.8}{
			\begin{tabular}{ccccc}
				\toprule[1pt]
				$G$ & $(\a,z)$ && $ \chi_1(\calM): 2^*{\rm ex} $  & $ \chi_2(\calM): 2^P{\rm ex}$ \\
				\midrule[1pt]

				$ \ZZ_{2^\ell}=\l a\r $ & $ (a,a^{2^{\ell-1}}) $ && $ 1 $ & $ 2-2^{\ell-1} $\\
				\midrule[0.5pt]
				
				$ \ZZ_{2^\ell}\times\ZZ_2=\l a,b\r $,\, \tabincell{c}{$ \ell=1 $\\ $ \ell\ge 2 $} & $ (a,b) $ && $ 2 $ & \tabincell{c}{$ 2 $ \\ $\times$} \\
				\midrule[0.5pt]
				
				\tabincell{c}{$ \ZZ_{2^\ell}{:}\ZZ_2=\l a\r{:}\l b\r $, \\ $ a^b=a^{2^{\ell-1}+1}$} & $ (a,b) $ && $ 2-2^{\ell-1} $ & $ \times$ \\
				\midrule[0.5pt]
				
				$ \D_{{2^{\ell+1}}}=\l a\r{:}\l b\r $ & \tabincell{c}{$ (a,b) $ \\ $ (ab,b) $} && \tabincell{c}{$ \times $ \\ $ 2 $} & \tabincell{c}{$ 2 $ \\ $ 2 $} \\
				\midrule[0.5pt]
				
				$ \SD_{{2^{\ell+1}}}=\l a\r{:}\l b\r $ & \tabincell{c}{$ (a,b) $ \\ $ (ab,b) $} && \tabincell{c}{$ \times $ \\ $ 2-2^{\ell-1} $} & \tabincell{c}{$ 2-2^{\ell-1} $ \\ $ 2-2^{\ell-1} $} \\
				\bottomrule[1pt]
		\end{tabular}}
	\end{table}
\end{proposition}

\begin{proof}
	By Lemma \ref{arc-map}, $ G $ is one of the groups appearing in Theorem \ref{thm:p-gps}. Now, the rotary pairs for such groups are determined in Lemma \ref{rotpair1}, \ref{rotpair2}. We therefore need only select, from all corresponding maps, those satisfying the condition $ 4\nmid\chi(\calM) $. Recall that, respectively for type $ 2^*{\rm ex} $ and $ 2^P{\rm ex} $ we have \[\mbox{$\chi_1(\calM)={|G|\over| \a|}-{|G|\over2}+{|G|\over|\l z,z^{\a}\r|}$, and $\chi_2(\calM)={|G|\over|\a|}-{|G|\over2}+{|G|\over| \a z|}$}.\] The abelian case $ G=\ZZ_{{2^\ell}} $ or $\ZZ_{{2^\ell}}\times\ZZ_2$ is simple, directly presented in the table. 
	
	For non-abelian $ G=\l a\r{:}\l b\r=\ZZ_{2^\ell}{:}\ZZ_2$, with $a^b=a^\lambda$, by Lemma \ref{rotpair2}, either (i) $ (\a,z)=(a,b) $, $ \lambda\in\{-1,2^{\ell-1}\pm 1\} $; or (ii) $ (\a,z)=(ab,b) $, $ \lambda\in\{-1,2^{\ell-1}-1\} $. Note that $ G $ is non-abelian, $ z\ne z^\a $, and so $ \l z,z^{\a}\r $ is a dihedral group of order $ 2|zz^{\a}| $.
    \begin{itemize}
    	\item[(i)] If $ (\a,z)=(a,b) $, then $ |\a|=|a|=2^{\ell} $, and $zz^{\a}=bb^a=(a^{-1})^ba= a^{1-\lambda} $, $ (\a z)^2=(ab)^2=aa^b=a^{1+\lambda} $, so respectively for $ \lambda=2^{\ell-1}+1,\, -1,$ or $ 2^{\ell-1}-1 $,\[\mbox{$ |zz^{\a}|= 2,\,2^{\ell-1} $ or $ 2^{\ell-1} $, and $ |\a z|= 2^{\ell},\,2 $ or $ 4 $.}\]
    	\item[(ii)] If $ (\a,z)=(ab,b) $, then $ zz^\a=bb^{ab}=(bb^a)^b $, and so \[\mbox{$ |\a|=|ab| $, $|zz^\a|=|(bb^a)^b|=|bb^a|$ and $ |\a z|=|a| $}\] are all given in (i).
    \end{itemize}
Then $\chi_1(\calM),\chi_2(\calM)$ are easily obtained, as shown in the table. In the case $\times$, $\chi(\calM)=4-2^{\ell}$ is divisible by $ 4 $ as $ \ell\ge2 $. In the case $\chi(\calM)=2-2^{\ell-1}$, we assume $ \ell\ge3 $, so that, $ 4\nmid2-2^{\ell-1} $.
\end{proof}

\noindent {\bf Acknowledgments.} The author is grateful to Prof.~Cai Heng Li (Southern  University of Science and Technology, Shenzhen) for the valuable ideas and suggestions.

\section*{Declarations}

The author declares that there is no any financial/personal relationship with other people/organizations not mentioned that can inappropriately influence the work.


\begin{thebibliography}{99}
	
\bibitem{Aut-directproduct} J.N. Bidwell, M.J. Curran and D. J. McCaughan, Automorphisms of direct products of finite groups, {\it Arch. Math. (Basel)} {\bf 86} (2006), 481-489.

\bibitem{Sylow-metac}
D. Chillag and J. Sonn, Sylow-metacyclic groups and Q-admissibility, {\it Isreal J. Math.} {\bf 40} (1981), 307-323.

\bibitem{CNS}
M. Conder, R. Nedela and J. \v Sir\'a\v n, Classification of regular maps of Euler characteristic $-3p$, {\it J. Combin. Theory Ser. B} {\bf 102} (2012), 967-981.

\bibitem{CPS}
M. Conder, P. Poto\v cnik and J. \v Sir\'a\v n, Regular maps with almost Sylow-cyclic automorphism groups, and classification of regular maps with Euler characteristic $-p^2$, {\it J. Algebra} {\bf 324} (2010), 2620-2635.

\bibitem{2-gp-split} M.J. Curran, The automorphism group of a split metacyclic 2-group, {\it Arch. Math. (Basel)} {\bf 89} (2007), 10-23.

\bibitem{Char-p}
A.B. d\'Azevedo, R. Nedela and J. \v Sir\'a\v n,
Classification of regular maps of negative prime Euler characteristic, {\it Trans. Amer. Math. Soc.} {\bf 357} (2005), 4175-4190.

\bibitem{bi-rotary}
A.B. d\'Azevedo, D.A. Catalano and J. \v Sir\'a\v n, Bi-rotary maps of negative prime characteristic, {\it Ann. Comb.} {\bf 23} (2019), 27-50.

\bibitem{GW}
J.E. Graver and M.E. Watkins, Locally finite, planar, edge-transitive graphs, {\it Mem. Amer. Math. Soc.} {\bf 126} (1997), no. 601, vi+75 pp.

\bibitem{HLZZ}
P.C. Hua, C.H. Li, J.B. Zhang and H. Zhou,
Regular maps with square-free Euler characteristic, {\it Comm. Algebra} {\bf 53} (2025), no. 6, 2266-2277.

\bibitem{HLZZ2}
P.C. Hua, C.H. Li, J.B. Zhang and H. Zhou,
Finite solvable groups of square-free Euler characteristic, in preparation.

\bibitem{unfaithful}
C. H. Li and J. \v Sir\'a\v n, Regular maps whose groups do not act faithfully
on vertices, edges, or faces, {\it Europ. J. Combinatorics} {\bf 26} (2005), 521-541.

\bibitem{Robinson}
D.J.S. Robinson, {\it A course in the theory of groups}, 2nd Ed., Graduate Texts in Mathematics 80, Springer-Verlag, New York, 1996.

\bibitem{Rotman}
J.J. Rotman, {\it An introduction to the theory of groups}, 4th Ed., Graduate Texts in Mathematics 148, Springer-Verlag, New York, 2024.

\bibitem{ST} J. \v Sir\'a\v n, T. Tucker and M.E. Watkins, Realizing finite edge-transitive orientable maps, {\it J. Graph Theory} {\bf 37} (2001), 1-34.

\end{thebibliography}
\end{document}